\documentclass[11pt,twoside]{amsart}

\usepackage{amsthm, amsmath, amsopn, latexsym, array, amssymb,enumitem,mdwlist}
\usepackage{dsfont}
\usepackage{eucal}
\usepackage[utf8]{inputenc}
\usepackage[english]{babel}
\usepackage[all]{xy}
\usepackage{hyperref}
\usepackage[margin=4.0cm]{geometry}
\usepackage{color}
\usepackage{graphicx}
\makeatletter
\newcommand*{\bigcdot}{}
\DeclareRobustCommand*{\bigcdot}{%
  \mathbin{\mathpalette\bigcdot@{}}%
}
\newcommand*{\bigcdot@scalefactor}{.5}
\newcommand*{\bigcdot@widthfactor}{2.5}
\newcommand*{\bigcdot@}[2]{%
  \sbox0{$#1\vcenter{}$}
  \sbox2{$#1\cdot\m@th$}%
  \hbox to \bigcdot@widthfactor\wd2{%
    \hfil
    \raise\ht0\hbox{%
      \scalebox{\bigcdot@scalefactor}{%
        \lower\ht0\hbox{$#1\bullet\m@th$}%
      }%
    }%
    \hfil
  }%
}
\makeatother

\usepackage{manyfoot}

\usepackage{hyperref}

\newtheorem{theorem}{Theorem}[section]
\newtheorem{maintheorem}{Theorem}
\newtheorem{maincorollary}[maintheorem]{Corollary}

\newtheorem*{theorem*}{Theorem}
\newtheorem{lemma}[theorem]{Lemma}
\newtheorem{proposition}[theorem]{Proposition}

\newtheorem{claim}[theorem]{Claim}

\theoremstyle{definition}
\newtheorem{remark}[theorem]{Remark}
\newtheorem*{remark*}{Remark}
\newtheorem{definition}[theorem]{Definition}

\newtheorem*{acknowledgements}{Acknowledgements}

\newcommand{\PP}{\mathds{P}}
\newcommand{\QQ}{\mathds{Q}}
\newcommand{\ZZ}{\mathds{Z}}
\newcommand{\NN}{\mathds{N}}
\newcommand{\kk}{\mathds{k}}
\newcommand{\cha}{\mathrm{char}}

\newcommand{\HH}{\mathrm{H}}

\newcommand {\coker}{\mathrm{Coker}}
\newcommand {\gr}{\mathrm{gr}}

\newcommand {\Hom}{\mathrm{Hom}}
\newcommand {\Ext}{\mathrm{Ext}}
\newcommand {\im}{\mathrm{Im}}
\newcommand {\Ker}{\mathrm{Ker}}
\newcommand {\rD}{\mathrm{D}}
\newcommand {\rd}{\mathrm{d}}
\newcommand {\rL}{\mathrm{L}}
\newcommand {\rR}{\mathrm{R}}

\newcommand {\id}{\mathrm{id}}

\newcommand {\Hilb}{\mathcal{H}\kern -0.25ex{\mathit ilb\/}}

\newcommand {\cA}{\mathcal{A}}
\newcommand {\cL}{\mathcal{L}}
\newcommand {\cU}{\mathcal{U}}
\newcommand {\cB}{\mathcal{B}}

\newcommand{\cV}{{\mathcal V}}
\newcommand{\cW}{{\mathcal W}}
\newcommand{\cE}{{\mathcal E}}
\newcommand{\cHom}{{\mathcal H}om}
\newcommand{\cExt}{{\mathcal E}xt}
\newcommand{\cH}{{\mathcal H}}
\newcommand{\cG}{{\mathcal G}}

\newcommand{\cF}{{\mathcal F}}

\newcommand{\cN}{{\mathcal N}}
\newcommand{\cO}{{\mathcal O}}
\newcommand{\cT}{{\mathcal T}}

\newcommand{\epi}{\twoheadrightarrow}

\newcommand\rp{{\mathrm{p}}}
\newcommand\ru{{\mathrm{u}}}

\makeatletter
\@namedef{subjclassname@1991}{Subject}
\@namedef{subjclassname@2000}{Subject}
\@namedef{subjclassname@2010}{Subject}
\@namedef{subjclassname@2020}{2020 Mathematics Subject Classification}
\makeatother

\title{Non-Ulrich representation type}



\author[D. Faenzi, F. Malaspina, G. Sanna]{Daniele Faenzi, Francesco Malaspina, Giangiacomo Sanna}
\thanks{F.M. partially supported by GNASA-INdAM and
  MIUR grant Dipartimenti di Eccellenza 2018-2022 (E11G18000350001).
  D.F.
  partially supported by ISITE-BFC project \textit{Motivic Invariants of
  Algebraic Varieties} ANR-lS-IDEX-OOOB and
  EIPHI ANR-17-EURE-0002.}

\begin{document}

\sloppy

\begin{abstract}
  We show a remarkable property of the CM-wild variety $\PP^1 \times
  \PP^2$, namely that the only ACM sheaves moving in positive-dimensional
  families are Ulrich bundles. A complete
  classification of the non-Ulrich range is given.

  We prove that this feature is unique in the sense that any 
  other ACM reduced closed subscheme $X \subset \PP^N$ of dimension $n
  \ge 1$ belongs to the well-known list of 
  CM-finite or CM-tame varieties, or else it remains CM-wild upon removing
  Ulrich sheaves.
\end{abstract}

\maketitle
\section{Introduction}

Given a reduced closed subscheme $X \subset \PP^N$ of dimension $n>0$
over an algebraically closed field $\kk$, we say that
$X$ is arithmetically Cohen-Macaulay (ACM) if its homogeneous
coordinate  algebra $\kk[X]$ is a graded Cohen-Macaulay ring.
A coherent sheaf $\cE$ on $X$ is
ACM if the module $E$ of global sections of
$\cE$ is a maximal Cohen-Macaulay (MCM) module over $\kk[X]$.

A few ACM varieties $X$ support only finitely many isomorphism classes of
indecomposable ACM sheaves (up to twist), so $X$ is of \textit{finite
CM representation type}, or \textit{CM-finite}. These varieties are
classified in
\cite{eisenbud-herzog:CM} and turn out to be: projective spaces,
smooth quadrics, rational
normal curves, the Veronese surface in $\PP^5$ and the rational
surface scroll of degree $3$ in $\PP^4$.

All ACM subvarieties $X$ besides these cases are CM-infinite.
In a few cases, $X$ supports only discrete families of non-isomorphic
indecomposable ACM sheaves. This happens for quadrics of corank $1$
and $\cha(\kk) \ne 2$ (see \cite[\S 4]{buchweitz-greuel-schreyer}) and
\textit{$t$-chains of rational curves}, i.e. $A_t$-configurations of
smooth rational projective curves, for $t \ge 2$ 
(see \cite{drozd-greuel}). We call these varieties
CM-discrete (although some authors call these varieties CM-finite as well).

Some CM-non-discrete varieties support at most $1$-dimensional families of
isomorphism classes of indecomposable ACM sheaves. In dimension $1$, this happens when
$X$ is an elliptic curve (by \cite{atiyah:elliptic}), or, in view of
\cite{drozd-greuel}, an $\tilde A_t$-configuration of
smooth rational projective curves for $t \ge 2$, or a rational projective curve with a
single simple node (which one may think of as an $\tilde
A_1$-configuration). These curves are called \textit{cycles of
  rational curves}.
In higher dimension, this happens when $X$ is a rational
surface scroll of degree $4$ in $\PP^5$, see \cite{faenzi-malaspina:minimal}.
A variety $X$ with this property is of \textit{tame CM-type}.

As opposed to the previous kinds of varieties, whose CM-categories are
under control, one introduces CM-wild varieties.
In terms of representation theory of
algebras, $X$ is of CM-wild type if the category of
finitely generated modules over any finitely generated associative $\kk$-algebra $\Lambda$
admits a representation embedding into the category of graded MCM
modules over $\kk[X]$. This means that there is an exact functor
$\Phi$ that carries finitely generated $\Lambda$-modules to graded MCM $\kk[X]$-modules
such that, given finitely generated $\Lambda$-modules $M$ and $M'$, $\Phi(M) \simeq
\Phi(M')$ implies $M \simeq M'$ and $\Phi(M)$ is indecomposable
whenever $M$ is. We refer e.g. to \cite[\S XIX]{simson-skowronski:III}
and \cite{drozd-greuel} for more
precise definitions of tame and wild representation type of algebras.

The category of ACM sheaves over a CM-wild variety is at least as
rich as the category of finitely generated modules of an arbitrary
finitely generated associative algebra.
It is clear that, if $X$ is CM-wild in the algebraic
sense, then $X$ supports families of
pairwise non-isomorphic indecomposable ACM sheaves of arbitrarily
large dimension, so $X$ is CM-wild in the geometric sense.

The main result of \cite{faenzi-pons:arxiv} asserts that all ACM integral
closed subschemes in $\PP^N$ which are not in the list of CM-finite,
CM-discrete or CM-tame varieties mentioned above are CM-wild in the algebraic sense.
\smallskip

Among ACM sheaves, a special role is played by Ulrich sheaves. These
are characterized by the linearity of the minimal graded free resolution
over the polynomial ring of their module of global sections.
Ulrich sheaves, originally studied for computing Chow forms,
conjecturally exist over any variety
(we refer to \cite{eisenbud-schreyer-weyman}). They are important for
Boij-Söderberg theory (cf. \cite{eisenbud-schreyer:betti-cohomology,shcreyer-eisenbud:ICM}) and
for the determination of the representation type of varieties
(see \cite{faenzi-pons:arxiv}).

Over many smooth algebraic varieties, heuristics about Ulrich sheaves point out
that, among ACM sheaves of a fixed rank, they frequently  move in the
largest families, i.e. the dimension of their deformation space is
maximal among such sheaves. For instance, Fano threefolds of Picard
number one and index at least two admit ACM sheaves of rank two;
most of them are semistable, and their moduli space has
the largest dimension precisely in the case of Ulrich sheaves (see \cite{brafa2}).
This happens also on some Fano threefolds of higher Picard rank
(we refer e.g. to \cite{casnati-faenzi-malaspina:3} and references therein).
The above considerations motivate the belief that, when $X$ is CM-wild,
there should exist families of pairwise non-isomorphic indecomposable
Ulrich sheaves of arbitrarily large dimension (so $X$ should be Ulrich
wild). Algebraically, it should be possible to construct the
representation embedding $\Phi$ in such a way that it lands into the
category of Ulrich (also known as \textit{maximally generated}) MCM modules.

\smallskip

The present paper is devoted to further study the impact of Ulrich
sheaves on the representation type of varieties, mostly on smooth ones. Namely,
taking for granted the slogan that Ulrich sheaves should move in the
largest families, we ask what happens if we exclude them: does the
representation type of $X$ change? More precisely, can $X$ be
downgraded to a CM-finite or CM-tame variety if we require that the
image of the representation embedding $\Phi$ contains only finitely
many Ulrich modules for any given rank, so that $X$ is algebraically
non-Ulrich CM-wild? In particular, are there unbounded families of
indecomposable ACM sheaves on $X$ which are not Ulrich?

Our first main contribution is that the answer to this question is
negative, except for the two smooth CM-tame surfaces and for a single
CM-wild variety, which is $\PP^1 \times \PP^2$. More specifically,
after excluding Ulrich sheaves, the
two rational scrolls of degree $4$ and $\PP^1 \times \PP^2$
become of finite
CM representation type, while all other varieties keep their
representation type unchanged.
This holds for all reduced ACM varieties of dimension $n \ge 1$ and
$\cha(\kk) \ne 2$.

\begin{maintheorem} \label{main1}
  Let $X \subset \PP^N$ be a reduced closed non-degenerate ACM
  subscheme of dimension
  $n \ge 1$. Then $X$ is algebraically non-Ulrich CM-wild unless $X$ is:
  \begin{enumerate}[label=\roman*)]
  \item a linear space;
  \item a quadric hypersurface of corank at most one;
  \item an $A_t$-configuration of smooth rational curves, for some $t \ge 1$;
  \item an $\tilde A_t$-configuration of smooth rational curves, for $t
    \ge 1$, or an elliptic curve;
  \item \label{tame} a surface scroll of degree $d$ in $\PP^{d+1}$
    with $d \in \{3,4\}$;
  \item \label{P1xP2} the Segre product $\PP^1 \times \PP^2$ in $\PP^5$.
  \end{enumerate}
In cases \ref{tame} and \ref{P1xP2}, $X$ supports
  only finitely many non-Ulrich ACM sheaves.
\end{maintheorem}

For the second theorem, we use suitably chosen sets of generators of
the derived category of coherent sheaves over projective bundles over
$\PP^1$ to obtain a complete
classification of the ACM indecomposable bundles (Ulrich or not) over
$\PP^1 \times \PP^2$ and quartic scrolls. This second case is actually a direct
extrapolation from \cite{faenzi-malaspina:minimal}, so the main point
is to treat $\PP^1 \times \PP^2$, embedded as a degree $3$ submanifold
of $\PP^5$ via the Segre product.
To state the result, let us introduce some notation. Consider the
projection $\pi$ from $X=\PP^1
\times \PP^2$ to $\PP^1$ and put $F$ for the divisor class of a fibre
of $\pi$ and $L$ for the pull-back of the class of a line on
$\PP^2$. Set $\Omega_\pi$ for the cotangent bundle of $\PP^2$,
pulled-back to $X$.

\begin{maintheorem} \label{B}
Let $\cF$ be an indecomposable ACM sheaf on $\PP^1 \times \PP^2$,
assume $\HH^0(\cF)=0$ and $\HH^0(\cF(1))\ne 0$. Then $\cF$ is:
\begin{enumerate}[label = \roman*)]
\item either an Ulrich bundle of the form:
  \begin{equation}
    \label{ulrich-extension}
  0 \to \cO_X(-F)^{\oplus a} \to \cF \to \cO_X(F-L)^{\oplus b} \to 0,
  \qquad \mbox{  for some $a,b \in \NN$},
\end{equation}
\item either $\cO_X(-1)$ or $\cO_X(-L)$ or the Ulrich bundle $\Omega_\pi(L)$.
\end{enumerate}
\end{maintheorem}

This has the following surprising corollaries.

\begin{maincorollary} \label{cor1}
  Given a polynomial $p \in \QQ[t]$, any non-empty moduli space of $H$-semistable
  ACM sheaves on $X$ with Hilbert polynomial $p$ is a finite set of points.
\end{maincorollary}

Put $c_0=0$, $c_1=1$, $c_{k+2}= 3c_{k+1}-c_k$ and $c_{-k}=c_k$ for all $k \ge
0$. The numbers $c_k$ are the odd terms of the Fibonacci sequence.

\begin{maincorollary} \label{cor2}
  For any $k \in \ZZ$ there is a unique indecomposable sheaf $\cU_k$
  fitting into:
  \[
  0 \to \cO_X(-F)^{\oplus c_{k-1}} \to \cU_k \to \cO_X(F-L)^{\oplus  c_{k}} \to 0.
  \]
  The sheaves $\cU_k$ are Ulrich and rigid, and satisfy:
  \[
  \cU_k^\vee
  \otimes \omega_X(2) \simeq \cU_{1-k}.
  \]
  Up to twist by $\cO_X(t)$, any
  rigid indecomposable ACM sheaf on $X$ is isomorphic either to
  $\cO_X$, or to $\cO_X(-L)$,
  or to $\Omega_\pi(L)$, or to $\cU_k$, for some $k$.
\end{maincorollary}

The paper is organized as follows. 
We start by recalling some basic notions and preparing the proof of
our main results in \S \ref{background}, where we quickly sketch how
to deal with the case of curves by following the literature.
In \S \ref{constructing} we provide a result ensuring the existence of
unbounded families of ACM non-Ulrich sheaves under certain
conditions; this is actually a slight modification of \cite[Theorem
  A]{faenzi-pons:arxiv}.
Sections \ref{higher degree}
and \ref{minimal degree}
are devoted to the proof of Theorem
\ref{main1} in dimension $2$ and higher, with the exception of the statement concerning $\PP^1
\times \PP^2$.
More specifically, \S \ref{minimal degree} proves it for varieties of minimal
degree (except for $\PP^1 \times \PP^2$), i.e. non-degenerate integral varieties $X \subset \PP^N$ of dimension $n
\ge 2$ and degree $d = N-n+1$, while \S \ref{higher degree} proves it when $d >
N-n+1$, the special case $n=2$ being treated in \S \ref{section:dP},
separately from the range $n \ge 3$ showing up in \S \ref{higher}.
Finally, in \S \ref{section:segre} we analyze ACM bundles on the exceptional case
mentioned above, namely the Segre product $\PP^1 \times \PP^2 \subset
\PP^5$. Theorem \ref{B} is proved in \S \ref{subsection:beilinson},
cf. in particular Theorem \ref{classifica}. The two corollaries above
are proved in \S \ref{subsection:rigid}.

\begin{acknowledgements}
  We are grateful to Gianfranco Casnati for invaluable help.
\end{acknowledgements}

\section{Basic facts} \label{background}

Let $\kk$ be a field. Given an integer $N$, set
$\PP^N$ for the projective space of hyperplanes through the origin of
$\kk^{N+1}$.

\subsection{Notation and conventions}

Let $X \subset \PP^N$
be a closed integral subscheme of dimension
$n$.
We assume throughout the paper that $X$ is non-degenerate, namely,
there is no hyperplane of $\PP^N$ that contains $X$.
The variety $X$ is equipped with the very ample line bundle $\cO_X(1)$
defined as the restriction of $\cO_{\PP^N}(1)$ via the embedding $X
\subset \PP^N$. We will write $H$ for the divisor
class of $\cO_X(1)$.

The coordinate ring $R$ of $\PP^N$ is the graded polynomial algebra in
$N+1$ variables with the standard grading, namely
$R=\kk[x_0,\ldots,x_N]$.
The homogeneous coordinate ring $\kk[X]$ is the graded algebra
$\kk[X]=R/I_X$, where $I_X$ is the homogeneous radical ideal of
polynomials vanishing on $X$.

The degree of $X$ is computed via the Hilbert polynomial of $I_X$. We will
be denoted it by $d$.

\subsection{Cohen-Macaulay and Ulrich conditions}

Given a coherent sheaf $\cE$ on $X$,
the $i$-th \textit{cohomology module} of $\cE$ is the $\kk[X]$-module:
\[
\HH^i_*(\cE)=\bigoplus_{k\in\ZZ}\HH^i(X,\cE\otimes\cO_X(k)).
\]

For $i \ge 1$, the $\kk[X]$-modules $\HH^i_*(\cE)$ are artinian.

\begin{definition}
A coherent sheaf $\cE$ on $X$  is called \textit{ACM}, standing for
\textit{Arithmetically Cohen-Macaulay}, if $\cE$ is locally
Cohen-Macaulay on $X$ and:
 \[
 \HH^i_*(\cE)=0, \qquad \mbox{$\forall i \in  \{1,\ldots,n-1\}$.}
 \]

 Equivalently, the
 minimal graded free resolution of the module of global sections
 $E=\HH^0_*(\cE)$, seen as $R$-module, has length $N-n$.
\end{definition}
A locally Cohen-Macaulay sheaf on a smooth scheme is locally free, we
also call it a \textit{vector bundle} or simply a \textit{bundle}.

\medskip

The variety $X$ itself is said to be ACM if $X$ is projectively
 normal and $\cO_X$ is ACM. This is equivalent to ask that $\kk[X]$ is
 a graded Cohen-Macaulay ring, which in turn amounts to the fact
 that the minimal graded free resolution of
 $\kk[X]$ as $R$-module has length $N-n$.
 In this case, the line bundles $\cO_X(k)$
 are ACM.

 \begin{definition} Let $d$ be the degree of the embedded variety $X
   \subset \PP^N$.
   For $r>0$, a rank-$r$ ACM sheaf $\cE$ on $X$ is said to be
   \textit{Ulrich} if there
   is $t \in \ZZ$ such that $\HH^0(X,\cE(t-1)) = 0$ and $h^0(X,\cE(t)) =
   rd$. We say that $\cE$ is \textit{initialized by $t$} (we omit ``by $t$''
   if $t=0$).
 \end{definition}

Given a coherent sheaf $\cE$ on $X$,  asking that $\cE$ is initialized
and Ulrich is
tantamount to $\HH^*(X,\cE(-j))=0$ for all $1 \le j \le n$, cf. \cite[Proposition
2.1]{eisenbud-schreyer:betti-cohomology}.

\begin{remark}
  We should warn the reader that the usual definition of Ulrich sheaf
  in the literature is equivalent to our definition of initialized
  Ulrich sheaf. We adopted this slightly different definition in order
  to work with sheaves which are Ulrich up to a twist.
\end{remark}

\subsection{Semistability}

Let $X \subset \PP^N$ be a closed subscheme of dimension $n>0$ embedded by the very
ample divisor $H$.
Stability of sheaves on $X$ will always mean Gieseker stability of
pure $n$-dimensional sheaves with respect to the divisor $H$.

The Hilbert polynomial of a coherent sheaf $\cE$ on
$X$, computed with respect to $H$, is denoted by $P(\cE,t)$.
The rank of $\cE$ is defined as the element $r \in \QQ$ such that
 the leading coefficient of $P(\cE,t)$ equals $r d/n!$. For $r \ne 0$, we write $\rp(\cE,t):=P(\cE,t)/r$ for the {\it reduced
   Hilbert polynomial} of $\cE$.

 Given polynomials $p,q \in \QQ[t]$, we write $p \preceq q$ if $p(t) \leq
 q(t)$ for $t \gg 0$.
 A coherent sheaf $\cE$ of rank $r \ne 0$ is semistable if it is
 pure (i.e. all its subsheaves have support of dimension $n$) and, for any
 non-zero subsheaf $\cF \subsetneq \cE$, we have $\rp(\cF,t) \preceq \rp(\cE,t)$.
 Stability is defined by strict inequalities.

 A coherent sheaf $\cE$ on $X$ is simple if $\Hom_X(\cE,\cE)$ is generated by $\id_\cE$.

\subsection{Basic remarks on Theorem \ref{main1}}

\label{remarks}

Here are some comments about Theorem \ref{main1} from the introduction.
Again, we assume that $X \subset \PP^N$ is a $n$-dimensional
closed subscheme over a field $\kk$, with $n \ge 1$.
\begin{remark} 
To obtain CM-wildness of $X$, it
  suffices to find a representation embedding of some algebra of wild
  representation type to the category of ACM sheaves on $X$ (cf. for
  instance \cite[\S 1.2]{faenzi-pons:arxiv}).
  We will
  mostly take such algebra to be the free $\kk$-algebra in
  two generators or the path algebra of the Kronecker quiver with
  two vertices and three arrows.
\end{remark} 
\begin{remark} 
  In the setting of Theorem \ref{main1} (i.e. $X \subset \PP^N$ is
  reduced closed, non-degenerate, ACM and
  $\kk$ is algebraically closed), the restriction $\cha(\kk)
  \ne 2$ is only needed to deal with 
  the case of quadric hypersurfaces of corank $1$, which is derived
  from \cite{buchweitz-greuel-schreyer}, so the result is valid also
  in characteristic $2$ except perhaps for this case. More information on MCM
  modules on quadrics in characteristic $2$ can be found in
  \cite{buchweitz-eisenbud-herzog}.

\end{remark} 
\begin{remark} 
 The statement for curves in Theorem \ref{main1} is a consequence of
  \cite{drozd-greuel}, cf. also \cite{burban-drozd-greuel,
    bodnarchuk-burban-drozd-greuel}. It should be pointed out that the
  cohomological vanishing required for a sheaf to be ACM plays no role
  in dimension $1$ so the statement is really about locally Cohen Macaulay
  sheaves.
  
  More in detail, given a reduced connected projective curve $X \subset \PP^N$
  of degree $d$ over an algebraically closed field $\kk$,
  in order to check that $X$ is non-Ulrich CM-wild it
  suffices to find a representation embedding of some wild 
  $\kk$-algebra into the category of vector bundles over $X$ in such a
  way that, for any given rank $r$, the resulting rank-$r$ bundles $\cE$ satisfy:
  \begin{equation}
    \label{nonU}
  \HH^0(\cE(-H))=0, \qquad \mbox{$0 \ne \dim \HH^0(\cE) \ne d r$},
  \end{equation}
  except for finitely many choices of $\cE$.
  
  In turn, this is already the case for all vector bundles appearing in
  \cite[Theorem 7]{bodnarchuk-burban-drozd-greuel} if $X$ has
  arithmetic genus $g > 1$. Indeed, such bundles have degree $r$ so
  \eqref{nonU} follows easily from Riemann-Roch.

  For curves of arithmetic genus $g \le 1$
  which are not of type $A_t$ or $\tilde A_t$, the construction of
  \cite{drozd-greuel} provides bundles $\cE$ whose pull-back under the
  normalization $\pi : \tilde X \to X$ decompose as a direct sum
  $\oplus_{i=1}^t \cE_i$, where $(\tilde X_i \mid i=1,\ldots,t)$ is
  the set of irreducible components of $\tilde X$ and,
  for each $i \in \{1,\ldots,t\}$, the sheaf $\cE_i$ is a vector bundle on $\tilde X_i$.
  One checks that the degrees of the bundles $(\cE_i\mid i=1,\ldots,t)$ can be chosen
  in such a way that $\pi^*(\cE)$, 
  and hence $\cE$, satisfy \eqref{nonU} by Riemann-Roch.

  In the same way, the freedom in the choice of $\deg(\cE_i)_i$ allows
  to define infinitely indecomposable vector bundles which are not
  isomorphic up to twist over $A_t$-configurations for $t \ge 2$,
  as well as 1-parameter family thereof over elliptic curves and 
  $\tilde A_t$-configurations for $t \ge 1$.
\end{remark}

In view of the previous remark, in the proof of Theorem \ref{main1} we
will be allowed to assume that the dimension $n$ of $X$ is at least $2$.

\section{Representation embeddings and non-Ulrich sheaves}

\label{constructing}

Let  $X \subset \PP^N$ be a non-degenerate
closed subscheme of dimension $n > 0$ over a field $\kk$.
We propose here a criterion, based on classical ideas about
extensions of sheaves and modules, for $X$ to be non-Ulrich CM-wild.
Since this does not really depend on $X$ being smooth or ACM, we
formulate it in a more general setting than what is actually needed to
prove Theorem \ref{main1}.
The result is a slight modification of \cite[Theorem
  A]{faenzi-pons:arxiv}.

\begin{theorem} \label{generale}
Let $\cA$ and $\cB$ be simple
semistable ACM sheaves such that $\rp(\cB) \prec \rp(\cA)$ and assume
$\dim_{\kk} \Ext^1_X(\cB,\cA) \ge 3$. Then the following holds:
\begin{enumerate}[label=\roman*)]
\item\label{wild} the subscheme $X$ is CM-wild;
\item\label{n>1} if $n \ge 2$ and $\cA$ and $\cB$ are not Ulrich
  initialized by the same integer, then $X$ is algebraically
  non-Ulrich CM-wild;
\item \label{n=1} the same conclusion as in \ref{n>1} holds also for
  $n=1$ if there is no $t \in \ZZ$ such that
  $\HH^0(X,\cA(t))=\HH^1(X,\cB(t))=0$.
\end{enumerate}
\end{theorem}

\begin{proof}
  We use the setting and notation of \cite[Theorem
  A]{faenzi-pons:arxiv}. To be in position of applying that result, we should
  verify that any non-zero morphism $\cA \to \cB$ is an
  isomorphism. But this is obvious since $\rp(\cB) \prec \rp(\cA)$ and
  $\cA$ and $\cB$ are semistable, so any morphism $\cA \to \cB$ is
  actually zero, so \ref{wild} is clear.

  Therefore $X$ is algebraically CM-wild. Assume now that no integer $t$ turns $\cA(t)$
  and $\cB(t)$ into initialized Ulrich sheaves.
  Recall that, by construction, the sheaves appearing in the families
  provided by \cite[Theorem
  A]{faenzi-pons:arxiv} are
  extensions of copies of $\cA$ and $\cB$. If a sheaf $\cE$ is an
  extension of say $a$ copies of $\cA$ and
  $b$ copies of $\cB$, it suffices to prove that $\cE$ is actually
  non-Ulrich, as soon as $a,b$ are both non-zero.

  To check this,  by contradiction we let $t$ be an integer that
  initializes $\cE$ as Ulrich sheaf, i.e. such that
  $\HH^*(X,\cE(t-j))=0$ for all $1 \le j \le n$. Since $\cA$ and
  $\cB$ are ACM by assumption, we have the vanishing $\HH^i(X,\cA(t-j))=\HH^i(X,\cB(t-j))$ for $1
  \le i \le n-1$ and for all $j \in \ZZ$.

  By definition, $\cE$ fits into an exact sequence of the form:
  \[
  0 \to \cA^{\oplus a} \to \cE \to \cB^{\oplus b} \to 0,
  \]
  where we may assume $a\ne 0 \ne b$. Therefore, from the vanishing $\HH^i(X,\cE(t-j))=0$ we deduce
  $\HH^0(X,\cA(t-j))=0=\HH^n(X,\cB(t-j))$ for $1 \le j \le n$.

  Now, if $n\ge 2$, because $\cA$ is ACM, the vanishing
  $\HH^1(X,\cA(t-j)))$ takes place for all $j \in \ZZ$ so we see that
  $\HH^0(X,\cE(t-j))=0$ implies $\HH^0(X,\cB(t-j))=0$ for $1 \le j \le
  n$. This implies that $\cB$ is Ulrich,
  initialized by $t$,
  and similarly we get that the holds true for $\cA$. But this is
  excluded, and  we conclude that \ref{n>1} holds.

  With the same setup we can prove also \ref{n=1}. Indeed, when $X$ is
  a curve, a coherent sheaf $\cE$ is Ulrich if and only if
  there is $t \in \ZZ$ such that $\HH^i(X,\cE(t))=0$ for all $i$, which
  implies $\HH^0(X,\cA(t)) = 0$ and $\HH^1(X,\cB(t)) = 0$.  But
  our assumption implies that there is no $t \in \ZZ$ such that $\HH^0(X,\cA(t))$ and $\HH^1(X,\cB(t))$
  vanish together, so $\cE$ is not Ulrich.
\end{proof}

\section{Varieties of non-minimal degree} \label{higher degree}

Let  $X \subset \PP^N$ be a non-degenerate
closed subscheme of dimension $n$ over an algebraically closed field $\kk$. Assume $X$ is
reduced and ACM.
Put $d = \deg(X)$.

We mentioned in \S \ref{remarks} that, in order to prove Theorem
\ref{main1}, we can assume $n \ge 2$.
In this section we would like to treat the case when $X$ is not of
\textit{minimal degree}, which is to say $d \ge N-n+2$.

We first look at the case $(n,d)=(2,N)$, so $X$ is a surface of
quasi-minimal degree, which we deal with in the next
paragraph. The remaining cases are basically already in
\cite{faenzi-pons:arxiv}, up to the the result, proved in \S
\ref{higher}, that the $c$-th syzygy of an ACM sheaf supported on a
$c$-codimensional linear section is an ACM sheaf which is never Ulrich
when $d \ge N-n+2$.

\subsection{Surfaces of quasi-minimal degree}
\label{section:dP}
For this subsection, $X \subset \PP^N$ is an ACM reduced closed
surface of degree $N$, so $X$ is of quasi-minimal degree.
It turns out that $X$ is locally Gorenstein, namely
$\omega_X \simeq \cO_X(-1)$, see \cite[Corollary 4.1.5]{migliore:liaison}.

Let us first observe that, if $X$ is reducible, then $X$ is non-Ulrich
CM-wild. Indeed, in view of \cite[\S 7.1]{faenzi-pons:arxiv} and since
$n \ge 2$, we are in position to apply \cite[Theorem
5.2]{faenzi-pons:arxiv} and show that $X$ is CM-wild. However, the
sheaves obtained in this way are not Ulrich. Indeed, such a sheaf $\cE$ appears as an
extension of an ACM sheaf $\cF_1(q)$, on a first component $X_1$ of
$X$, for some $q > 0$, and the structure sheaf of a second component $X_2$ of $X$, and
the resulting sheaf $\cE$ is not Ulrich for $q \gg 0$.
So we may assume until the end of the section that $X$ is integral.

\subsubsection{Syzygies of Ulrich bundles}

Let $X$ be an integral ACM surface of quasi-minimal degree.
We first assume that $X$ is not a cone.
In view of \cite[\S 7.4]{faenzi-pons:arxiv}, there exist stable
initialized Ulrich 
bundles $\cE_1$ and $\cE_2$ of rank $2$ on $X$ and determinant $\cO_X(2)$, such that:
\begin{equation} \label{Eij}
  \begin{aligned}
  & \Hom_X(\cE_i,\cE_i)=\kk \id_{\cE_i}, && \mbox{for $i \in \{1,2\}$,} \\
  & \Hom_X(\cE_i,\cE_j)=0, && \mbox{if $\{i,j\}=\{1,2\}$,} \\
  & \dim \Ext^1_X(\cE_i,\cE_i) = 5, && \mbox{for $i \in \{1,2\}$,} \\
  & \dim \Ext^1_X(\cE_i,\cE_j) = 4, && \mbox{if $\{i,j\}=\{1,2\}$,} \\
  & \Ext^p_X(\cE_i,\cE_j) = 0, && \mbox{for $i,j \in \{1,2\}$ and $p
    \ge 2$.}
\end{aligned}
\end{equation}

Consider the evaluations of global sections, for $i \in \{1,2\}$:
\[
  e_{\cE_i} : \HH^0(\cE_i) \otimes \cO_X \to \cE_i.
\]
We define $\cA_i=\ker(e_{\cE_i})^\vee$, for $i \in \{1,2\}$.

\begin{proposition}
  The sheaves $\cA_1$ and $\cA_2$ are simple, ACM vector bundles which are not Ulrich. They
  satisfy:
  \begin{equation}
    \label{Aij}
    \begin{aligned}
      & \Hom_X(\cA_i,\cA_i)=\kk \id_{\cA_i}, && \mbox{for $i \in \{1,2\}$,} \\
      & \Hom_X(\cA_i,\cA_j)=0, && \mbox{if $\{i,j\}=\{1,2\}$,} \\
      & \dim \Ext^1_X(\cA_i,\cA_i) = 5, && \mbox{for $i \in \{1,2\}$,} \\
      & \dim \Ext^1_X(\cA_i,\cA_j) = 4, && \mbox{if $\{i,j\}=\{1,2\}$,} \\
      & \Ext^p_X(\cA_i,\cA_j) = 0, && \mbox{for $i,j \in \{1,2\}$ and $p
        \ge 2$.}
    \end{aligned}
  \end{equation}
\end{proposition}

\begin{proof}
It is well-known that the module of global sections of an initialized
Ulrich bundle of rank $r$ over $X$ is generated by $Nr$ elements of
degree $0$. Therefore the evaluation maps
$e_{\cE_k}$ are surjective for $k \in \{1,2\}$ and the sheaves $\ker(e_{\cE_k})$  are locally free and ACM of rank $2(N-1)$ on $X$.
In other words, $\HH^1_*(\cA_k^\vee)=0$ for $k \in \{1,2\}$.
Also, by definition of $e_{\cE_k}$, we have $\HH^0(\cA_k^\vee)=0$ for $k \in \{1,2\}$.

Using the isomorphism $\omega_X \simeq \cO_X(-1)$ and Serre duality we get that
$\cA_k$ are ACM bundles on $X$ for $k \in \{1,2\}$.
Also, $\cE_k^\vee \simeq  \cE_k(-2)$, so we have:
\begin{equation}
  \label{EA}
  0 \to \cE_k(-2) \to \cO_X^{\oplus 2N} \to \cA_k \to 0.  
\end{equation}

This gives at once $\HH^0(\cA_k(-1))=0$ and $\dim(\HH^0(\cA_k)) = 2N <
2N(N-1)$ hence $\cA_k$ is not Ulrich, for $k \in \{1,2\}$.

For $i,j,k \in \{1,2\}$, since $\cE_j$ is an initialized Ulrich bundle, we
have $\HH^*(\cE_j(-2))=0$, thus applying $\Hom_X(-,\cE_j(-2))$ to \eqref{EA}
we get $\Ext^p_X(\cE_k,\cE_j) \simeq \Ext^{p+1}_X(\cA_k,\cE_j(-2))$
for all $p \ge 0$. Using Serre duality and the vanishing we already
proved for $\cA_i^\vee$ we get $\Ext_X^*(\cA_i,\cO_X)=0$. Therefore, applying
$\Hom_X(\cA_i,-)$ to \eqref{EA}, we get
$\Ext^{p+1}_X(\cA_i,\cE_k(-2)) \simeq \Ext^p_X(\cA_i,\cA_k)$ for all $p
\ge 0$.
Summing up we get:
\[
  \Ext^p_X(\cE_i,\cE_j) \simeq \Ext^p_X(\cA_i,\cA_j), \qquad \mbox{for
    all $p$ and all $i,j \in \{1,2\}$}.
\]
Hence \eqref{Aij} follows from \eqref{Eij} and the proposition is proved.
\end{proof}

\medskip

If $X$ is a cone, then we use the construction of
\cite[\S 7.3]{faenzi-pons:arxiv}.
This gives two initialized Ulrich sheaves $\cE_1$ and $\cE_2$ of rank
$1$ on $X$ such that:
\[
  \begin{aligned}
  & \Hom_X(\cE_i^\vee,\cE_i^\vee)=\kk \id_{\cE_i}, && \mbox{for $i \in \{1,2\}$,} \\
  & \Hom_X(\cE_i^\vee,\cE_j^\vee)=0, && \mbox{if $\{i,j\}=\{1,2\}$,} \\
  & \dim \Ext^1_X(\cE_i^\vee,\cE_i^\vee) = N+1, && \mbox{for $i \in \{1,2\}$,} \\
  & \dim \Ext^1_X(\cE_i^\vee,\cE_j^\vee) = N, && \mbox{if $\{i,j\}=\{1,2\}$.}
\end{aligned}
\]

Since $\cExt^1_X(\cE_i,\cO_X)=0$ for $i \in \{1,2\}$, this allows to define two reflexive ACM sheaves $\cA_1$ and
$\cA_2$ of rank $(N-1)$ as in the previous proposition and show that
$\cA_1$ and $\cA_2$ are not Ulrich 
as $\HH^0(\cA_k(-1))=0$ and $\dim(\HH^0(\cA_k)) = N <N(N-1)$.
Note that, for $i\in \{1,2\}$, we have $\HH^*(\cE_i^\vee)=0$ by Serre
duality since 
$\omega_X \simeq \cO_X(-1)$ and $\cE_i$ is initialized Ulrich.
Again, $\Ext_X^p(\cA_i,\cO_X)=0$ for all $p$ and $i \in \{1,2\}$ so,
by the same argument as before, we get:
\[
    \begin{aligned}
      & \Hom_X(\cA_i,\cA_i)=\kk \id_{\cA_i}, && \mbox{for $i \in \{1,2\}$,} \\
      & \Hom_X(\cA_i,\cA_j)=0, && \mbox{if $\{i,j\}=\{1,2\}$,} \\
      & \dim \Ext^1_X(\cA_i,\cA_i) = N+1, && \mbox{for $i \in \{1,2\}$,} \\
      & \dim \Ext^1_X(\cA_i,\cA_j) = N, && \mbox{if $\{i,j\}=\{1,2\}$.} 
    \end{aligned}
\]

Summing up, independently on whether $X$ is a cone or not, in view \cite[Theorem A]{faenzi-pons:arxiv},
we get that $X$ is non-Ulrich CM-wild and even strictly CM-wild.

\subsubsection{A second construction for del Pezzo surfaces}

Let us give a second construction, with the further assumption that $X$
is smooth, so $X$ is an anticanonically embedded del Pezzo surface.
This construction has the advantage of being self-contained and the
drawback of relying on the explicit description of $X$ as a blown-up plane or
a quadric surface.
More precisely, recall that $X$ is either a blow-up of $\PP^2$ at $9-d$ points or the
product variety $\PP^1 \times \PP^1$.
We construct ACM bundles (Ulrich or not) on $X$ with the same methods
in both cases, only with a slightly different choice of the
invariants.

If $X$ is a blow-up of $\PP^2$, we fix
a birational surjective morphism $\pi : X \to \PP^2$ and let
$\cO_X(L)=\pi^*(\cO_{\PP^2}(1))$, $M=2L$.
Given $(a,b) \in \NN^2$, with $a \ge 2$,
 we put $D(a,b)=3ab-a^2-b^2+1$ and $b_a=2a$.
In the second case we set $\pi_1$ and
$\pi_2$ to be the projection maps onto the two $\PP^1$ factors and let
$\cO_X(L)=\pi_1^*(\cO_{\PP^1}(1))$ and $\cO_X(F)=\pi_2^*(\cO_{\PP^1}(1))$.
This time we take $(a,b) \in \NN^2$ with $a \ge 1$ and we put $D(a,b)=4ab-a^2-b^2+1$, $b_a=3a$, $M=2L+F$.

\begin{proposition}
Choose $(a,b)$ so that $D(a,b)>0$
and $b \ge b_a$. Then, for $f$ general enough in
$\Hom_X(\cO_X(L)^{\oplus b},\cO_X(M)^{\oplus a})$, the sheaf $\cE=\ker(f)$ is simple, locally
free and ACM, with $\dim \Ext^1_X(\cE,\cE)=D(a,b)$ and
$\Ext^2_X(\cE,\cE)=0$; $\cE$ is not Ulrich
when $b>b_a$.
\end{proposition}

\begin{proof}
  Note that $b \ge b_a\ge a$ and that the locally free sheaf
  $\cH=\cHom_X(\cO_X(L)^{\oplus b},\cO_X(M)^{\oplus a}) \simeq
  \cO_X(M-L)^{\oplus b a}$ is globally generated.
  Therefore, for a general choice of $f \in \HH^0(\cH)$ and for any
  integer $k \in \{0,\ldots,a-1\}$, the degeneracy locus $D_k(f)$
  defined by the $(k+1)$-minors of the
  associated map $f : \cO_X(L)^{\oplus b} \to \cO_X(M)^{\oplus a}$ has
  codimension $(b-k)(a-k)$ in $X$ in view of a Bertini-type result, see for instance
  \cite[Teorema 2.8]{ottaviani:codim}, or \cite[Lemma
  11.6]{andrade-jardim-correa}.
  In particular, for $k=a-1$, since $b-a+1 \ge b_a-a+1 \ge 2a-a+1 \ge
  a+1 \ge 3$ and $\dim(X)=2$, we have $D_{a-1}(f)=\emptyset$ so $f$ is surjective.

  Then, the sheaf $\cE =
  \ker(f)$ is locally free of rank $b-a \ge 2$. We write down the exact sequence:
  \begin{equation}
    \label{LM}
  0 \to \cE \to \cO_X(L)^{\oplus b} \to \cO_X(M)^{\oplus a} \to 0.
  \end{equation}

  Next, observe  that the $\kk$-vector space
  $\Hom_X(\cO_X(L),\cO_X(M))$ has dimension $3$ or $4$ depending on
  whether $X$ is a blow-up of $\PP^2$ or $X \simeq \PP^1 \times \PP^1$.
  In both cases, the assumption $D(a,b)>0$ ensures that \cite[Theorem
  4]{kac} applies (cf. the argument of \cite[Proposition 3.5
  (i)]{costa-miro_roig-pons_llopis}) and shows that $\cE$ is simple if
  $f$ is general enough. The same argument proves $\dim
  \Ext^1_X(\cE,\cE)=D(a,b)$ and $\Ext^2_X(\cE,\cE)=0$.

  Next, we show that $\cE$ is ACM.   Note that $\cO_X(L)$ is ACM for
  the polarization $H$ and that $\HH^0(\cO_X(M+t H))=0$ for any integer
  $t \le -1$, so \eqref{LM} gives $\HH^1(\cE(t H))=0$ for all $t \le -1$.
  Also, Serre duality gives $\HH^k(\cO_X(L-H))=\HH^k(\cO_X(M-H))=0$
  for all $k$. So once we make sure that
  $\HH^1(\cE)=0$, we will get that $\cE(H)$ is $H$-regular and hence
  $\HH^1(\cE(t H))=0$ for $t \ge 0$, so that $\cE$ will be proved to be ACM.

  So let us prove that $\HH^1(\cE)=0$. If $X$ is a blow-up of $\PP^2$, this
  follows from \cite[Propositions 1.1 and  4.1]{ellia-hirschowitz:ouest} in view of the assumption $b \ge b_a$.
  When $X \simeq \PP^1 \times \PP^1$, first we note that the condition $\HH^1(\cE)=0$
  is open
  on flat families, and that:
  \[
    \{\ker(f) \mid \mbox {$f \in \HH^0(\cH)$ gives a surjective map $\cO_X(L)^{\oplus b} \to \cO_X(M)^{\oplus a}$} \}
  \]
  defines a family of vector bundles on $X$ which is indeed flat.

  In view of this discussion, in order to get the statement for
  general $f \in \HH^0(\cH)$, it suffices
  to prove it for one
  choice of $f_0 \in \HH^0(\cH)$, provided that the associated $f_0 :
  \cO_X(L)^{\oplus b} \to \cO_X(M)^{\oplus a}$ is surjective.
  To choose a convenient element $f_0$, note that again a Bertini-type
  argument ensures that, for a general choice of $g \in
  \Hom_X(\cO_X(L)^{\oplus 3}, \cO_X(M))$, the map $g$ is surjective.
  Define $\cF=\ker(g)$, so:
  \begin{equation}
    \label{LM0}
  0 \to \cF \to \cO_X(L)^{\oplus 3} \to \cO_X(M) \to 0.
  \end{equation}
  Then, choose $f_0$ to be a diagonal map consisting of $a$ copies of $g$
  as above, completed by $b-b_a=b-3a$ zeroes. We get thus a surjective
  map $f_0:\cO_X(L)^{\oplus b} \to \cO_X(M)^{\oplus a}$ and $\cE_0 = \ker(f_0)\simeq \cO_X(L)^{b-3a} \oplus
  \cF^{\oplus a}$.

  We still have to prove $\HH^1(\cE_0)=0$. To do it, it suffices to
  show $\HH^1(\cF)=0$. In turn, we use an
  argument analogous to
  \cite[Proposition 5.9]{eisenbud-schreyer-weyman} to show this and
  actually prove that $\cF$ is Ulrich on $(X,H)$, where $H=2L+2F$. Indeed,
  $c_1(\cF)=L-F$ so $\cF\simeq \cF^\vee(L-F)$ and the dual of
  \eqref{LM0} yields the exact
  sequence:
  \[
    0 \to \cO_X(-L-2F)  \to \cO_X(-F)^{\oplus 3} \to \cF \to 0.
  \]
  This implies immediately $\HH^*(\cF)=0$. Also,
  \eqref{LM0} gives:
  \[
  0 \to \cF(-H) \to \cO_X(-L-2F)^{\oplus 3} \to \cO_X(-F) \to 0.
  \]
  which implies $\HH^*(\cF(-H))=0$, so that $\cF$ is Ulrich.

  We have thus proved that $\cE$ is ACM. Finally, $\cE$ is not Ulrich
  as soon as $b>b_a$. Indeed, from \eqref{LM} we have
  $\HH^0(\cE(-H))=0$ so we have that $\cE$ is not Ulrich as soon as we
  show:
  \[
    0 < \chi(\cE) < d(b-a).
  \]
  
    Now, on one hand, the assumption $b>b_a$ guarantees
    $\chi(\cE)>0$. On the other hand, when $X$ is a blow-up of $\PP^2$
    we get $\chi(\cE)=3(b-2a)<3(b-a) \le d(b-a)$, while for $X \simeq
    \PP^1 \times \PP^1$ (hence $d=8$) we have
    $\chi(\cE)=2(b-4a)<2(b-a) \le 8(b-a)$. In both cases the desired
    equality holds and the statement is proved.
\end{proof}

\begin{remark}
  The previous proof  actually implies that, for $X \simeq \PP^1
  \times \PP^1$ embedded by $2L+2F$ and $b=3a$,
  the sheaf $\cE$ obtained by $f$ as in the previous
  proposition is a simple
  Ulrich bundle of rank $2a$.
\end{remark}

The previous proposition shows the geometric version of Theorem \ref{main1} when $X$ is a del
Pezzo surface. Indeed, choosing for instance $b=b_a+1$, the above construction
provides families of non-Ulrich pairwise non-isomorphic indecomposable bundles,
whose dimension grows as a quadratic
function of $a$.

\subsection{The higher range}

\label{higher}

According to \S\ref{section:dP}, we have to justify that Theorem \ref{main1} holds for all non-degenerate
reduced closed ACM subschemes $X \subset \PP^N$ of dimension $n \ge 2$
and degree $d$, provided that $(d,n,N)$ lie in the \textit{higher
  range}, namely $d \ge N-n+3$, or $n\ge 3$ and $d = N-n+2$.
We essentially extract this from
\cite{faenzi-pons:arxiv} up to the result, proved below, that syzygies
of Ulrich sheaves are never Ulrich.

Indeed,  \cite[Theorem
4.2]{faenzi-pons:arxiv} already asserts that, when a subscheme $X$ of
dimension at least two as
above is in the range $d \ge N-n+2$ (i.e. $X$ is
not of \textit{minimal degree}), then it is of wild CM-type.

What we do here is to show that the proof of \cite[Theorem
4.2]{faenzi-pons:arxiv} already yields non-Ulrich sheaves.
Looking into this proof we see that it
proceeds by reduction to a transverse linear section $Y$ of $X$ of
dimension one in case $d \ge N-n+3$, or of dimension two in case $d =
N-n+2$. Namely, setting $c$ for the codimension of $Y$ in $X$, one
first constructs families of arbitrarily large dimension of
indecomposable pairwise non-isomorphic ACM sheaves $\cF$ on $Y$.
Then, one considers a minimal graded free resolution $F_{\bigcdot} \to F$ of the finitely
generated $\kk[X]$-module
$F=\HH^0_*(\cF)$, of the form:
\begin{equation}
  \label{resF}
0 \leftarrow F \leftarrow F_0 \leftarrow F_1 \leftarrow \cdots \leftarrow F_{\ell-1} \xleftarrow{d_\ell} F_\ell \leftarrow \cdots
\end{equation}
where, for all $i \ge 0$, $F_i$ is a finitely-generated free
$\kk[X]$-module.
For $i \ge 0$, we write  $\Sigma_i^X(F)=\im(d_i)$ and $\Sigma_i^X(\cF)$ for the sheafification of $\Sigma_i^X(F)$.

The sheaf $\cE=\Sigma_c^X(\cF)$ is ACM over $X$.
One shows via \cite[Theorem B]{faenzi-pons:arxiv} that the families
of sheaves $\cE$ constructed in this way are still made-up of
indecomposable pairwise non-isomorphic ACM sheaves, provided that the
sheaves $\cF$ are actually Ulrich.

Having this in mind, we consider the following setup.
Let $X \subset \PP^N$ be a non-degenerate reduced closed ACM subscheme of
dimension $n$ and degree $d$ over a field $\kk$. Let $M \subset \PP^N$ be a
linear subspace of codimension $c \ge 1$ and assume
 $Y=X \cap M$ is of dimension $n-c \ge 1$.
 Let  $\cF$ be an ACM sheaf on $Y$. 
 By  \S\ref{section:dP} and the discussion above, Theorem \ref{main1} holds for $d \ge N-n+2$, once
we prove the following result.

\begin{theorem}
The sheaf $\cE=\Sigma^X_c(\cF)$ is Ulrich on $X$ if and only if $\cF$
is Ulrich on $Y$ and $d=N-n+1$.
\end{theorem}

\begin{proof}
  If $\cF$ is Ulrich on $Y$ and $d=N-n+1$ then $\cE$ is Ulrich on $X$,
  according to \cite[Proposition 3.6]{faenzi-pons:arxiv}. What we have
  to prove is that $\cE$ is not Ulrich on $X$ if either $\cF$ is not
  Ulrich on $Y$, or $d\ge N-n+2$.
  
  Put $R_X=\kk[X]$ and $R_Y=\kk[Y]$. We consider $F=\HH^0_*(\cF)$ as a
  graded $R_X$-module.
  Recall that $F$ is a graded Cohen-Macaulay module on $R_X$, which is
  non-maximal as its dimension is $\dim(R_Y) = n-c+1 < n+1 = \dim(R_X)$.
  Without loss of generality, we may assume that $\cF$ is initialized.
  So, the
  $R_X$-module $F$ is generated by finitely many elements of positive degree.
  Therefore, the free modules $(F_i \mid i \ge 0)$ appearing in the
  minimal graded free resolution $F_{\bigcdot}$ of $F$
  over $R_X$ take the form:
  \[
    F_i = \bigoplus_{j \in \{i,\ldots,r_i\}} R_X(-j)^{\oplus a_{i,j}},
  \]
  for some sequence of integers $(r_i \mid i \in \NN)$ with $r_i \ge
  i$ for all $i \in \NN$, and some
  uniquely determined integers $(a_{i,j} \mid i \in \NN, j 
  \in \{i,\ldots,r_i\})$, called the $R_X$-Betti numbers of $F$.
  \smallskip
  
  Write the Koszul resolution $K_{\bigcdot}$ of $R_Y$ over $R_X$ as:
  \[
    0 \leftarrow R_Y \leftarrow K_0 \leftarrow \cdots \leftarrow K_c
    \leftarrow 0,
    \qquad \mbox{with
      $K_i = R_X(-i)^{\oplus {c \choose i}}$ for all $i \in \{0,\ldots,c\}$}.
  \]

  Since $\cF$ is an initialized ACM sheaf over $Y$, there exists an injective map $\varphi : R_Y \to F$. This
  map lifts to a map of graded complexes
  $\varphi_{\bigcdot} : K_{\bigcdot} \to F_{\bigcdot}$ which we write,
  for $i \in \{0,\ldots,c\}$ as:
  \begin{equation}
    \label{phi_i}
    \varphi_i : R_X(-i)^{\oplus {c \choose i}}
    \to \bigoplus_{j \in \{i,\ldots,r_i\}} R_X(-j)^{\oplus a_{i,j}}.
  \end{equation}
  In particular, the polynomial maps
  $\varphi_0,\ldots,\varphi_c$ are actually constant.

  Next, we observe that $\varphi_i$ is injective for all
  $i \in \{0,\ldots,c\}$. Indeed, $\varphi_0$ is injective since
  $\varphi \ne 0$ and $K_0=R_X$.
  For $i \in \{1,\ldots,c\}$,
  by induction on $i$ we may assume that $\varphi_{i-1}$ is injective
  so that the induced map $\Sigma_i^X(R_Y) \to \Sigma_i^X(F)$ is also
  injective. Thus $\ker(\varphi_i)$ is contained in
  $\Sigma_{i+1}^X(R_Y)$. But looking at \eqref{phi_i} we see that
  $\ker(\varphi_i)$ is generated by elements of degree $i$, while all
  elements of $\Sigma_{i+1}^X(R_Y)$ have degree at least
  $i+1$. Therefore $\ker(\varphi_i)=0$.

  In view of the previous paragraph, for each $i \in \{0,\ldots,c\}$
  we have a splitting $F_i \simeq K_i \oplus G_i$ for some graded
  $R_X$-module $G_i$ such that for each $i \in \{1,\ldots,c\}$ the differential
  $d_i : F_i \to F_{i-1}$ is upper triangular according to the
  block-matrix form:
  \begin{equation}
    \label{block}
    d_i : K_i \oplus G_i \to K_{i-1} \oplus G_{i-1},
  \end{equation}
  i.e. $K_i \to K_{i-1} \oplus G_{i-1}$ factors
  through the Koszul differential $K_i \to K_{i-1}$.

  Thus the map $\varphi_{c-1}$ induces an
  injection:
  \[
    R_X(-c)=\Sigma_c^X(R_Y) \hookrightarrow \Sigma_c^X(F)=E.
  \]
  This says that $\HH^0(X,\cE(c)) \ne 0$.
  \smallskip
  
  Next, use \cite[Sequence (3.2)]{faenzi-pons:arxiv} to get a long exact sequence:
  \[
    0 \leftarrow \Hom_{R_Y}(F,R_Y(1)) \leftarrow
    E^\vee(1-c) \leftarrow F^\vee_{c-1}(1-c) \leftarrow \cdots
    \leftarrow  F^\vee_0(1-c)  \leftarrow 0.
  \]
  Set $Q_0=F_0$ and, for $i \in \{1,\ldots,c-1\}$, define:
  \[
    Q_i= \coker(F^\vee_{i-1}(1-c) \to F^\vee_i(1-c)),
  \]
  where the maps are extracted from the above complex. We get an
  injection $Q_{c-1} \to E^\vee(1-c)$. Note that the above maps have a
  block-matrix form which is the transpose of \eqref{block} and that
  the cokernel of the transpose of the Koszul differential
  $K^\vee_{c-2}(-c) \to K^\vee_{c-1}(-c)$ is the homogeneous ideal $I_{Y/X}$
  of $Y$ in $X$. Therefore, the map
  $F^\vee_{c-2}(1-c) \to F^\vee_{c-1}(1-c)$ commutes with restricting
  the source to $K^\vee_{c-2}(1-c)$ and the target to
  $K^\vee_{c-1}(1-c)$. This gives an induced surjection $Q_{c-1} \to
  I_{Y/X}(1)$. Then, we may extract a non-trivial map:
  \[
    R_X^{\oplus c}  = K^\vee_{c-1}(1-c)\to Q_{c-1},
  \]
  which composes to a non-trivial map $R_X^{\oplus c} \to
  E^\vee(1-c)$. So $\HH^0(X,\cE^\vee(1-c)) \ne 0$.

  \medskip
  Having set up all this, we can prove that, if $d \ge N-n+2$, no integer $t$ turns $\cE$ into
  an initialized Ulrich sheaf. Indeed, let $t$ be such an integer, so
  that $\HH^*(X,\cE(t-j))=0$ for $1 \le j \le n$. We proved
  $\HH^0(X,\cE(c)) \ne 0$ and hence $t \le c$.

  On the other hand, by
  \cite[Lemma 3.1]{faenzi-pons:arxiv}
  we have, since $X$ is not of minimal degree,
  that $\HH^0(X,\omega_X(n-1)) \ne 0$.
  So there is an injective map
  $\cO_X \to \omega_X(n-1)$.
  Tensoring $\cE^\vee(1-c)$ with this map and using
   $\HH^0(X,\cE^\vee(1-c)) \ne 0$, we get
  $\HH^0(X,\cE^\vee \otimes \omega_X(n-c)) \ne 0$.
  By Serre duality we have thus $\HH^n(X,\cE(c-n)) \ne 0$, which
  implies $t \ge c+1$. Thus no integer $t$ turns $\cE$ into an
  initialized Ulrich sheaf.

  \medskip
  
  Finally we prove that, if $\cF$ is not an initialized Ulrich sheaf
  on $Y$, then again no integer $t$ turns $\cE(t)$ into an initialized
  Ulrich sheaf on $X$. By the previous argument we have $t \le c$ and
  we only need to prove $\HH^n(X,\cE(c-n)) \ne 0$.
  Now, if $\cF$ is ACM
  but not Ulrich over the $(n-c)$-dimensional variety $Y$,
  since $\HH^0(Y,\cF(-1)) = 0 \ne \HH^0(Y,\cF)$,
  we must have $\HH^{n-c}(Y,\cF(c-n)) \ne 0$.
  Sheafifying the resolution $F_{\bigcdot} \to F$ and taking cohomology,
  we get injections, for $k \in \{1,\ldots,c\}$:
  \[
    \HH^{n-c}(Y,\cF(c-n)) \hookrightarrow \cdots \hookrightarrow  \HH^{n-c+k}(X,\Sigma^X_k(\cF)(c-n)) \hookrightarrow \cdots \hookrightarrow \HH^{n}(X,\cE(c-n)).
  \]

  Therefore $\HH^n(X,\cE(c-n)) \ne 0$. This concludes the proof.
\end{proof}

\section{Varieties of minimal degree}

\label{minimal degree}

For this section, let $X \subset \PP^N$ be a non-degenerate
closed subscheme of dimension $n \ge 2$ over an algebraically closed
field $\kk$.
Assume that $X$ is ACM of degree $d = N-n+1$, i.e., $X$ is of \textit{minimal degree}.
We want to prove Theorem \ref{main1} for $X$.

\medskip

Assume first that $X$ is not integral (not
reduced, or reduced but not irreducible). Then the argument of \cite[\S
6]{faenzi-pons:arxiv} applies to show that $X$ is non-Ulrich
CM-wild. This follows from the fact that $X$ supports ACM sheaves
supported on the different components of $X$, or on the reduced
structure of a component of $X$, and that we can twist the leftmost
term of this extension by $\cO_X(q)$ for $q \gg 0$ and therefore
obtain a representation embedding whose source is a wild algebra and
whose target consists of ACM sheaves on $X$ which are not Ulrich.

\medskip
If $X$ is integral but not smooth, then $X$ is a cone over a smooth
variety which is again ACM and of minimal degree, whose apex is a linear space
of dimension $m \ge 0$. Also in this case,
the argument of \cite[\S 6.1]{faenzi-pons:arxiv} gives rise to ACM
sheaves $\cE$ which are not Ulrich except for finitely many choices of $\cE$. Indeed, these sheaves arise as extensions of the form:
\[
  0 \to \cA(q) \to \cE \to \cB \to 0,
\]
where $\cA$, $\cB$ are sheaves obtained extending to $X$ the
presentation of initialized Ulrich sheaves on the base of the cone and
$q > 0$. For $q \gg 0$, the sheaf $\cE$ cannot be Ulrich, so we get that $X$ is non-Ulrich CM-wild.

\medskip

So from now on, in order to prove Theorem \ref{main1}, we can assume
that $X$ is smooth and irreducible, hence, according to the Bertini-del
Pezzo's classification (cf. \cite{eisenbud-harris:minimal-degree}), $X$ is a quadric hypersurface
or a rational normal scroll.

After setting up some notation, this case will be settled in the next
lemma.
Given $n \ge 2$ and a non-decreasing sequence $a=(a_1,\ldots,a_n)$ of integers $1 \le a_1 \le \cdots
\le a_n$ put $d=\sum_{i=1}^n a_i$ and $N=d+n-1$. We denote by
$S(a)=S(a_1,\ldots,a_n)$ the rational normal scroll defined as the
projectivization of $\oplus_{i=1}^n
\cO_{\PP^1}(a_i)$, embedded as a submanifold of degree $d$ in
$\PP^{d+n-1}$ by the tautological relatively ample line bundle.
We set $H$ for the hyperplane class of $S(a)$ and $F$ for the
class of a fibre of the projection $S(a) \to \PP^1$. Let
$\cL=\cO_X((d-1)F-H)$.

\bigskip

It is well-known that $X=S(a)$ is CM-finite if:
\[
  a \in \{(1,1), (1,2)\}.
\]
By \cite{faenzi-malaspina:minimal}, $X$ is CM-tame if:
\[
  a \in \{(1,3), (2,2)\}.
\]

We know by \cite{miro_roig:scrolls} that the rational normal scroll
$X=S(a)$ is Ulrich-wild except for the cases above:
\[
  a \in \{(1,1), (1,2), (1,3), (2,2)\}.
\]

If we seek unbounded families of non-Ulrich bundles, we should be a
bit more careful and exclude the exceptional cases appearing in the
statement of Theorem \ref{main1}, which is to say:
\[
  a \in \{(1,3), (2,2), (1,1,1)\}.
\]

So we actually
assume from now on:
\[
  n\ge 4, \qquad \mbox{or $n = 3$, $d \ge 4$,}
  \qquad \mbox{or $n = 2$, $d \ge 5$.}
\]

We start by noting that the setup of
\cite[Theorem B]{faenzi-malaspina:minimal} applies in any
dimension to give rigid Ulrich
bundles on $X$ associated with Fibonacci-like sequences.
Indeed, put $\ell = (n-1)d-n$, so that $h^i(\cO_X(H-d F)) = 0$
for all $i \ne 1$ and hence by Riemann-Roch:
\[
\ell = -\chi(\cL^\vee(-F)) = h^1(\cO_X(H-d F)) \ge 2,
\]
in our range for $(d,n)$.
Define recursively the Fibonacci-like numbers $a_{\ell,k} \in \NN$ by:
\[
a_{\ell,0}=0, \qquad a_{\ell,1}=1, \qquad a_{\ell,k+2}=\ell
a_{\ell,k+1}-a_{\ell,k}, \qquad \forall k \in \NN.
\]
Since $\ell \ge 2$, the sequence $(a_{\ell,k})$ is strictly increasing
along $k$.
\medskip

Recall the notion of exceptional sheaf $\cE$ on $X$, namely $\cE$ is a
simple coherent sheaf such that $\Ext^i_X(\cE,\cE)=0$ for $i > 0$.
Recall also that two exceptional sheaves $(\cE,\cF)$ form an
exceptional pair if $\Ext^i_X(\cF,\cE)=0$ for all $i$.
The pair  $(\cL,\cO_X(-F))$ is exceptional. We mentioned that
 $h^1(\cL^\vee(-F))=\ell$ and $h^i(\cL^\vee(-F))=0$ for $i \ne 1$,

Then, making use of mutations like in \cite[\S 2]{faenzi-malaspina:minimal}, we get that for each $k \ge 0$ there is a
unique exceptional sheaf $\cU_k$ which fits into:
\begin{equation}
  \label{Uk}
0 \to \cO_X(-F)^{\oplus a_{\ell,k}} \to \cU_k \to \cL^{\oplus a_{\ell,k+1}} \to 0.
\end{equation}

Theorem \ref{main1} will be proved for $X$ if check the following result.
\begin{lemma}
  The sheaves $\cB=\cU_k$ and $\cA = \cO_X$ satisfy the assumptions of
  Theorem \ref{generale} as soon as we choose:
  \begin{itemize}
  \item $k=0$ for $n \ge 4$;
  \item $k=1$ for $n=3$ and $d\ge 4$;
  \item $k=3$ for $n=2$ and $d \ge 5$.
  \end{itemize}
\end{lemma}

\begin{proof}
Working as in \cite[\S 2]{faenzi-malaspina:minimal} we check that $\cU_k$ is an exceptional Ulrich bundle which is actually
initialized by $t=1$. As a consequence, $\cU_k$ is (strictly)
semistable simple sheaf with:
\[
\rp(\cU_k)=\frac {td}{n!}\prod_{i=1}^{n-1} (t+i) \prec \frac
{td+n}{n!}\prod_{i=1}^{n-1} (t+i) =  \rp(\cO_X).
\]

The sheaf $\cO_X$ is not Ulrich and is obviously stable. So,
in order to verify the assumptions of Theorem \ref{generale},
it only remains to check the condition on the dimension of the extension space.
We note that $h^i(\cL^\vee)=0$ for $i \ge 2$ and that, by Riemann-Roch:
\[
\chi(\cL^\vee)=2n+(1-n)d.
\]
Looking at \eqref{Uk}, we deduce $h^i(\cU_k^\vee)=0$ for $i \ge 2$ so:
\begin{align}
\label{dimext} \dim_{\kk} \Ext^1_X(\cU_k,\cO_X) & = h^1(\cU_k^\vee) \ge-\chi(\cU_k^\vee)= \\
\nonumber & = -a_{\ell,k+1} \chi(\cL^\vee) - 2a_{\ell,k} = \\
\nonumber & = a_{\ell,k+1}((n-1)d-2n)  - 2a_{\ell,k}.
\end{align}

Assume now $n \ge 4$. In particular, we have $d\ge 4$. Recall that for the case $n \ge 4$ we have chosen $k=0$.
Hence we consider $\cU_0=\cL$ and note that $h^1(\cL^\vee)=-\chi(\cL^\vee) \ge 4$
because $(n-1) d - 2 n \ge 3 d - 8 \ge  4$. Thus the
lemma is proved for $n \ge 4$.

\smallskip

Next, assume $n=3$, so that our choice is $k=1$. Then, formula \eqref{dimext} gives $h^1(\cU_1^\vee) \ge 4
d^{2}-18 d+16$, which is at least $8$ as soon as $d \ge 4$.

\smallskip

Finally, let $n=2$, in which case we have chosen $k=3$. Then, evaluating 
 \eqref{dimext} we get $h^1(\cU_3^\vee) \ge d^{4}-10\,d^{3}+32\,d^{2}-36\,d+10$.
This is at least $5$ for $d \ge 5$.

\end{proof}

The proof of Theorem \ref{main1} is now complete.

\section{The Segre product of a line and a plane}
\label{section:segre}

Let us now turn to the analysis of the Segre product
 $X=\PP^1 \times \PP^2$, which we consider as a smooth submanifold of
$\PP^5$. In other words, $X$ is the rational normal scroll
$X=S(1,1,1)$ of degree $d=3$ embedded by the tautological relatively
ample divisor $H$, hence $X$ has minimal degree. A smooth hyperplane section
of $X$ is the CM-finite cubic scroll $S(1,2)$.

Our goal here is to classify all ACM indecomposable bundles on $X$. Of
course, this is not quite possible since Ulrich bundles form a wild
class in terms of representation theory, so we focus on non-Ulrich
bundles and we classify all those.

\subsection{A first classification result via homological non-vanishing}

Let us first give the basic ACM bundles that will be the output of the
classification. Put $\pi$ for the projection $X \to \PP^1$ and
$\Omega_\pi$ for the relative cotangent bundle. Here $X$ is a product
so $\Omega_\pi$ is the pull-back of the cotangent bundle of $\PP^2$ via
the projection $\sigma : X \to \PP^2$. Set $L=H-F$, so $\cO_X(L)=\sigma^*(\cO_{\PP^2}(1))$.
Recall:
\[\omega_X \simeq \cO_X(-2F-3L).\]

We easily see that $\cO_X(L)$ is ACM and $\Omega_\pi(H+L)$ is Ulrich.
We start with a lemma, inspired on \cite{ballico-malaspina:splitting}, that classifies these sheaves as bundles with a specific
non-vanishing.

\begin{lemma} \label{Omega}
  Let $\cE$ be a locally free sheaf on $X$.
  Then  $\cE \simeq \Omega_\pi(L)$ if and only
  if $\cE$ is indecomposable and:
    \begin{equation}
      \label{nonvan}
    \HH^1(\cE)=\HH^1(\cE(-1))=\HH^2(\cE(-2))=0, \qquad \HH^1(\cE(-L)) \ne 0.
    \end{equation}
\end{lemma}

\begin{proof}
  One implication is clear, so we assume that $\cE$ is an
  indecomposable locally free
  sheaf satisfying \eqref{nonvan} and we prove that $\cE \simeq \Omega_\pi(L)$.
  Recall the standard isomorphism $\Ext^1_X(\cO_X(L),\cE) \simeq
  \HH^1(\cE(-L))$. Then, write the vertical Euler sequence:
  \begin{equation}
    \label{euler}
  0 \to \Omega_\pi(L) \to \cO_X^{\oplus 3} \to \cO_X(L) \to 0,
  \end{equation}
  and apply $\Hom_X(-,\cE)$ to it. Since
  $\Ext^1_X(\cO_X,\cE)=\HH^1(\cE)=0$, we get a surjection:
  \[
  \Hom_X(\Omega_\pi(L),\cE) \epi \Ext^1_X(\cO_X(L),\cE) \simeq \HH^1(\cE(-L)).
  \]
  Take $e \in \HH^1(\cE(-L)) \setminus \{0\}$ and consider a map $f :
  \Omega_\pi(L) \to \cE$ lying in the preimage of $e$ under the above
  surjection.
  \medskip

  Further, we consider the dual vertical Euler sequence, written in
  the form:
  \begin{equation}
    \label{eulerdual}
  0 \to \cO_X(-2L)  \to \cO_X(-L)^{\oplus 3} \to \Omega_\pi(L) \to 0.
  \end{equation}
  Note that, by Serre duality, our assumption gives:
  \[
  \Ext^1_X(\cE,\cO_X(-L)) \simeq \HH^2(\cE(-2))^\vee=0.
  \]
  Next, we write the horizontal Euler sequence in the form:
  \begin{equation}
    \label{eulerh}
  0 \to \cO_X(-2L-2F)  \to \cO_X(-2L-F)^{\oplus 2} \to \cO_X(-2L) \to 0.
  \end{equation}

  Again our assumption gives, via Serre duality:
  \begin{align*}
  &\Ext^2_X(\cE,\cO_X(-2L-F)) \simeq \HH^1(\cE(-1))^\vee=0.
  \end{align*}
  We have thus a surjection as composition of surjections:
  \begin{equation}
    \label{surietta}
  \Hom_X(\cE,\Omega_\pi(L)) \epi \Ext^1_X(\cE,\cO_X(-2L)) \epi
  \Ext^2_X(\cE,\cO_X(-2)).
  \end{equation}

  Choose a generator $k_X$ of the vector space $\HH^3(\omega_X)$ and
  $h \in \Ext^2_X(\cE,\cO_X(-2))$ such that the Yoneda product
  \[
  \HH^1(\cE(-L)) \otimes \Ext^2_X(\cE,\cO_X(-2)) \to \HH^3(\cO_X(-2F-3L))
  \simeq \HH^3(\omega_X)
  \]
  sends $e \otimes h$ to $k$.
  Choose then $g :
  \cE \to \Omega_\pi(L)$ lying in the preimage of $h$ under the
  surjection \eqref{surietta}.
  \medskip

  It is well-known that $\Omega_{\PP^2}$ is a simple sheaf so the same
   holds for $\Omega_\pi(L)$. Therefore, as
  soon as the map $g \circ f$ is non-zero it must be a non-zero
  multiple of the identity.
  This implies immediately that $\Omega_\pi(L)$ is a direct summand of $\cE$, which
  forces $\cE \simeq \Omega_\pi(L)$ because $\cE$ is indecomposable.

  \medskip
  It remains to check that $g \circ f \ne 0$. To do this, we consider
  the following commutative diagram of Yoneda maps.
  \begin{equation}
    \label{diagrammozzo}
  \xymatrix@-2ex{
    \Hom_X(\Omega_\pi(L),\cE) \otimes \Hom_X(\cE, \Omega_\pi(L)) \ar[r] \ar[d]   & \ar[d] \Hom_X(\Omega_\pi(L), \Omega_\pi(L)) \\
    \Hom_X(\Omega_\pi(L),\cE) \otimes  \Ext^2_X(\cE,\cO_X(-2)) \ar[r] \ar[d]   & \Ext^2_X(\Omega_\pi(L),\cO_X(-2)) \ar[d]\\
    \Ext^1_X(\cO_X(L),\cE) \otimes  \Ext^2_X(\cE,\cO_X(-2)) \ar[r]    & \HH^3(\omega_X)
  }
  \end{equation}

  Our goal is to prove that the map appearing in the top row sends
  $f \otimes g$ to a non-zero element. The upper map in the left
  column sends $f \otimes g$ to $f \otimes h$, so it suffices to check
  that the map in the middle row sends $f \otimes h$ to a non-zero element.
  In turn, the lower map in the left
  column sends $f \otimes h$ to $e \otimes h$, so it is enough to show
  that the map in the bottom row sends $e \otimes h$ to a non-zero element.
  But this last map sends $e \otimes h$ to $k_X$, hence we are done.
\end{proof}

In a similar vein we show the following.

\begin{lemma} \label{uovo di colombo}
Let $\cE$ be an indecomposable locally free sheaf on $X$. Then:
\begin{enumerate}[label=\roman*)]
\item \label{isOL} there is an isomorphism $\cE \simeq \cO_X(-L)$ if and only if:
\begin{equation}
    \label{OL}
    \HH^0(\cE)=    \HH^1(\cE(-L))=    \HH^2(\cE(-F-2L))=0, \qquad
    \HH^0(\cE(L)) \ne 0.
  \end{equation}
\item \label{isO} there is an isomorphism $\cE \simeq \cO_X(-1)$ if and only if:
  \begin{equation}
    \label{O}
    \HH^0(\cE(L))=    \HH^1(\cE(-F))=    \HH^2(\cE(-1))=0, \qquad
    \HH^0(\cE(1)) \ne 0.
  \end{equation}
\end{enumerate}
\end{lemma}

\begin{proof}
  Both items have an obvious implication, what we have to prove is
  that $\cE$ is isomorphic to the desired sheaf after assuming the
  cohomological conditions.

  Let us prove \ref{isOL}. Choose a non-zero element $f$ of
  $\HH^0(\cE(L)) \simeq \Hom_X(\cO_X(-L),\cE)$.
  Next, we choose a generator $k_X$ of $\HH^3(\omega_X)$ and note that by
  Serre duality there exists $h \in
  \Ext^3_X(\cE,\cO_X(-2F-4L))$ such that the Yoneda pairing
  \[
  \Hom_X(\cO_X(-L),\cE) \otimes \Ext^3_X(\cE,\cO_X(-2F-4L)) \to
   \HH^3(\omega_X)
  \]
  sends $f \otimes h$ to $k$.

  Next, write again the exact sequences \eqref{euler},
  \eqref{eulerdual} and \eqref{eulerh}, twisted by lines bundles on
  $X$ so that they take the following form:
  \begin{align*}
    & 0 \to \cO_X(-2F-4L) \to \cO_X(-2F-3L)^{\oplus 3} \to \Omega_\pi(-2F-L) \to 0, \\
    & 0 \to \Omega_\pi(-2F-L) \to \cO_X(-2)^{\oplus 3} \to  \cO_X(-2F-L) \to 0, \\
    & 0 \to  \cO_X(-2F-L) \to \cO_X(-1)^{\oplus 2} \to \cO_X(-L) \to 0.
  \end{align*}
  We remark that the vanishing assumptions of \ref{isOL}
  and Serre duality imply:
  \begin{align*}
    & \Ext_X^3(\cE,\cO_X(-2F-3L)) \simeq \HH^0(\cE)^*=0, \\
    & \Ext_X^2(\cE,\cO_X(-2)) \simeq \HH^1(\cE(-L))^*=0, \\
    & \Ext_X^1(\cE,\cO_X(-1)) \simeq \HH^2(\cE(-F-2L))^*=0.
  \end{align*}
  Therefore, applying $\Hom_X(\cE,-)$ to the three sequences above we
  get a surjection:
  \begin{equation}
    \label{natasuriezione}
  \Hom_X(\cE,\cO_X(-L)) \epi \Ext^3_X(\cE,\cO_X( 2F-4L)).
  \end{equation}
  We choose now $g \in \Hom_X(\cE,\cO_X(-L))$ in the preimage of $h$.

  Therefore we have a commutative diagram of the form:
  \[
  \xymatrix@-3ex{
    \Hom_X(\cO_X(-L),\cE) \otimes \Hom_X(\cE,\cO_X(-L)) \ar[d] \ar[r]
    & \Hom_X(\cO_X(-L),\cO_X(-L)) \ar[d]\\
    \Hom_X(\cO_X(-L),\cE) \otimes \Ext^1_X(\cE,\cO_X(-2F-L)) \ar[d]
    \ar[r] & \Ext^1_X(\cO_X(-L),\cO_X(-2F-L)) \ar[d]\\
    \Hom_X(\cO_X(-L),\cE) \otimes \Ext^2_X(\cE,\Omega_\pi(-2F-L))
    \ar[d] \ar[r] & \Ext^2_X(\cO_X(-L),\Omega_\pi(-2F)) \ar[d]\\
    \Hom_X(\cO_X(-L),\cE) \otimes \Ext^3_X(\cE,\cO_X( 2F-4L)) \ar[r] & \Ext^3(\cO_X(-L),\cO_X(-2F-4L))
  }
  \]
  where the horizontal maps are given by the Yoneda pairing, the
  left vertical ones are given by the factorization of the map
  \eqref{natasuriezione} while the maps in the right vertical column
  are obtained by applying $\Hom_X(\cO_X(-L),-)$ to the three
  exact sequences above. Since the identity map of $\cO_X(-L)$ is sent
  to $k_X$ via the composition of vertical maps by construction, it follows that $g \circ f$ is sent
  to the identity of $\cO_X(-L)$ via the top horizontal map. This says
  that $\cO_X(-L)$ is a direct summand of $\cE$, and therefore proves
  $\cE \simeq \cO_X(-L)$ by the indecomposability of $\cE$.

  \medskip
  The proof of \ref{isO} is similar, so we only sketch the
  argument. The strategy this time is to apply $\Hom_X(\cE,-)$ to
  the exact sequences:
  \begin{align*}
    & 0 \to \cO_X(-3F-4L) \to \cO_X(-2F-4L)^{\oplus 2} \to \cO_X(-F-4L) \to 0, \\
    & 0 \to \cO_X(-F-4L) \to \cO_X(-F-3L)^{\oplus 3} \to  \Omega_\pi(-1) \to 0, \\
    & 0 \to  \Omega_\pi(-1) \to \cO_X(-F-2L)^{\oplus 3} \to \cO_X(-1) \to 0,
  \end{align*}
  and to use Serre duality which gives, via the assumption of
  \ref{isO}:
  \begin{align*}
    & \Ext_X^3(\cE,\cO_X(-2F-4L)) \simeq \HH^0(\cE(L))^*=0, \\
    & \Ext_X^2(\cE,\cO_X(-F-3L)) \simeq \HH^1(\cE(-F))^*=0, \\
    & \Ext_X^1(\cE,\cO_X(-F-2L)) \simeq \HH^2(\cE(-1))^*=0.
  \end{align*}
  The rest of the proof follows the same pattern as in \ref{isOL}.
\end{proof}

\subsection{Beilinson-type spectral sequence}

We use the derived category $\rD(X)$ of bounded complexes of coherent
sheaves over the smooth projective variety $X$, in order to write the
Beilinson-type spectral sequence associated with a coherent sheaf
$\cE$ on $X$ after fixing a convenient full exceptional sequence in $\rD(X)$.
Indeed, the point is that the terms of this spectral sequence take a
special form when $\cE$ is ACM, and this will be our basic tool to
classify such sheaves.

\subsubsection{Background on exceptional objects and mutations}

Let us first recall some terminology. An object $\cE$ of $\rD(X)$ is
called \textit{exceptional} if $\Ext_X^\bullet(\cE,\cE) = \kk$,
concentrated in degree zero.
An ordered set of exceptional objects $(\cE_0, \ldots, \cE_s)$ is called an \textit{exceptional collection} if $\Ext_X^\bullet(\cE_i,\cE_j) = 0$ for $i > j$.
An exceptional collection is \textit{full} when
$\Ext_X^\bullet(\cE_i,\cF) = 0$ for all $i$ implies $\cF =
0$. Equivalently, the collection is full when $\Ext_X^\bullet(\cF, \cE_i) = 0$
implies $\cF =
0$.

Exceptional collections can be \textit{mutated}, let us recall what
that means. Let $\cE$ be an exceptional object in $\rD(X)$.
Then there are endofunctors $\rL_{\cE}$ and $\rR_{\cE}$ of $\rD(X)$, called
respectively the \textit{left} and \textit{right mutation functors}
such that, for all
$\cF$ in $\rD(X)$ there are functorial distinguished triangles:
\begin{align*}
  &\rL_{\cE}(\cF) \to \Ext_X^\bullet(\cE,\cF) \otimes \cE \to \cF \to  \rL_{\cE}(\cF)[1], \\
  & \rR_{\cE}(\cF)[-1] \to \cF \to \Ext_X^\bullet(\cF,\cE)^\vee \otimes \cE \to \rR_{\cE}(\cF).	
\end{align*}

For all $i=0,\ldots,s$ we define the \textit{right} and \textit{left
  dual} objects:
\begin{align*}
  & \cE_i^{\rd} = \rL_{\cE_0} \rL_{\cE_1} \cdots \rL_{\cE_{s-i-1}} \cE_{s-i},\\
  & {}^\rd\cE_i = \rR_{\cE_s} \rR_{\cE_{s-1}} \cdots \rR_{\cE_{s-i+1}} \cE_{s-i}.
\end{align*}
It turns out that, if $\cE_\bullet = (\cE_0, \ldots, \cE_s)$ is a full exceptional
collection, then both $(\cE_0^{\rd},\ldots,\cE_s^{\rd})$ and
 ${}^\rd\cE_\bullet=({}^{\rd}\cE_0,\ldots,{}^{\rd}\cE_s)$ also are full exceptional
collections, called respectively the \emph{right} and \emph{left dual} collections
of $(\cE_0, \ldots, \cE_s)$. We refer to \cite[\S 2.6]{gorodentsev-kuleshov}.
The dual collections are characterized by the following property:
\begin{equation}\label{eq:dual characterization}
\Ext_X^\ell({}^\rd \cE_i, \cE_j) \simeq \Ext^\ell_X(\cE_i, \cE_j^\rd) \simeq
\left\{
\begin{array}{ll}
\kk, & \textrm{\quad if $i+j = s$ and $i = \ell$,} \\
0, & \textrm{\quad otherwise}.
\end{array}
\right.
\end{equation}

Given an object $\cF$ of $\rD(X)$ and a full exceptional
collection $(\cE_0, \ldots, \cE_s)$, there is a spectral sequence:
\[
\bigoplus_{r+t=q} \Ext^{r}_X({}^\rd\cE_{s-p}, \cF) \otimes \cH^t(\cE_{p}) = E_1^{p,q} \Rightarrow \cH^{p+q-s}(\cF),
\]
where $\cH^i$ denotes the $i$-th homology sheaf of $\cF$. This means
that, for all $(p,q)$ such that $p+q \ne s$ we have
$E_\infty^{p,q}=0$, while:
\[
\bigoplus_{p+q=s}E_\infty^{p,q} \simeq \gr(\cF),
\]
where $\gr(\cF)$
denotes the graded object with respect to a filtration of
$\cF$ of the form:
\begin{align*}
& \cF = \cF_0 \supset \cF_1 \supset \cdots \supset \cF_s \supset \cF_{s+1}=0, &&
                                                              \mbox{with:}
  && \cF_j/\cF_{j+1} \simeq E_\infty^{j,s-j}.
\end{align*}
The $r$-th differential of the
  spectral sequence reads
  $\delta_r^{p,q} : E_r^{p,q} \to
  E_r^{p+r,q-r+1}$.

  \subsubsection{An exceptional collection adapted to ACM sheaves}

  Let us choose a full exceptional collection over $X$ adapted to the
  classification of ACM sheaves.
  Recall that we denoted by $F$ the divisor class of the $\PP^2$-bundle map $\pi : X \to \PP^1$ so that $\rD(X)$ has the following semiorthogonal decomposition: 
  \begin{align*}
\rD(X)=&\langle\pi^*\rD(\PP^1)\otimes\cO_X(-2L),\pi^*\rD(\PP^1)\otimes\cO_X(-L),
        \pi^*\rD(\PP^1)\rangle= \\
       =&\langle\cO_X(-F-2L),\cO_X(-2L),\cO_X(-F-L),\cO_X(-L),\cO_X(-F),\cO_X\rangle.
  \end{align*}

Twisting by a line bundle $\cO_X(L)$ and performing a right mutation given by the Euler sequences on $\PP^1$, this is replaced by:
\[
\rD(X) = \langle\cO_X(-L),\cO_X(F-L),\cO_X(-F),\cO_X,\cO_X(L-F),\cO_X(L)\rangle.
\]

Since $\cO_X$ and $\cO_X(L-F)$ are mutually orthogonal, mutation gives:
\[
\rD(X) = \langle(\cO_X(-L),\cO_X(F-L),\cO_X(-F),\cO_X(L-F),\cO_X,\cO_X(L)\rangle.
\]

Finally, a right mutation given by the Euler sequences on $\PP^2$
gives the following full exceptional collection of vector bundles over $X$.
\[
\cE_\bullet=  \left(\cO_X(-L),\cO_X(F-L),\cO_X(-F),\cO_X(L-F),\Omega_\pi(L),\cO_X\right).
\]

Setting $\cT_\pi = \Omega_\pi^\vee \simeq \Omega_\pi(3L)$, we write the left dual of this collection as:
\begin{equation}
  \label{base}
^\rd\cE_\bullet = \left(\cO_X,\cO_X(L),\cO_X(1)[1],\cT_\pi(F)[1],\cO_X(F+2L)[2],\cO_X(2)[2]\right),
\end{equation}

Note that $\cE_1=\cO_X(F-L)$ is the Ulrich line bundle $\cL$ from \S \ref{minimal degree}.
By Künneth's formula one gets another special feature of
this collection, namely that:
\begin{equation}
  \label{nomaps}
  \begin{aligned}
    &\Hom_X(\cE_0,\cE_2) =  \Hom_X(\cE_0,\cE_3)= 0, \\
    &\Hom_X(\cE_1,\cE_j)=0, && \mbox{for $j \ne 1$}, \\
    &\Hom_X(\cE_2,\cE_4) = 0, \\
    &\Hom_X(\cE_3,\cE_4) = 0.
  \end{aligned}
\end{equation}

\subsection{Beilinson resolution of non-Ulrich sheaves}
\label{subsection:beilinson}

Our goal for this subsection is to prove the next result.

\begin{theorem} \label{classifica}
  Up to twist by $\cO_X(t)$, an indecomposable ACM bundle
  $\cF$ on $X$ is either isomorphic to $\cO_X(-1)$, or
  to $\cO_X(-L)$ or to $\Omega_\pi(L)$, or to
  an Ulrich bundle $\cU$ fitting into:
  \[
  0 \to \cO_X(-F)^{\oplus a} \to \cU \to \cO_X(F-L)^{\oplus  b} \to
0, \qquad \mbox{for some $(a,b) \in \NN^2$}.
  \]

\end{theorem}

The words ``up to a twist'' have the following
more precise meaning: up to replacing $\cF$ with $\cF(t)$ we may assume
that $h^0(\cF)=0$ and $h^0(\cF(1))\ne 0$. Then $\cF$ is exactly one
of the sheaves appearing in the statement of Theorem \ref{classifica}.
In other words, Theorem \ref{classifica} proves Theorem \ref{B} from
the introduction.
\bigskip

We will prove the theorem through several claims.
The very first argument is to use Lemma \ref{Omega}. Note that the
vanishing conditions appearing in that lemma are verified for any
twist of $\cF$ since $\cF$ is ACM, so if there is a twist $t \in \ZZ$ such
that $\HH^1(\cF(t H-L))\ne 0$, we will have $\cF(t) \simeq
\Omega_\pi(L)$. Theorem \ref{classifica} is proved in this case.

\smallskip

Therefore, from now on we may assume $\HH^1(\cF(t H-L)) = 0$ for all $t
\in \ZZ$.
The next step is to observe that, since $\cF$ locally free and $\cO_X(1)$ is
very ample, there is a unique $t_0 \in \ZZ$ such that $\cF(t_0)$ satisfies
$h^0(\cF(t_0))=0$ and $h^0(\cF(t_0+1))\ne 0$. We implicitly replace $\cF$
with $\cF(t_0)$ from now on. In particular we have $\HH^0(X,\cF)=0$. We put:
\[
a_{i,j}=\dim_{\kk} \Ext^i_X({}^\rd\cE_{j},\cF).
\]

\begin{claim}
  Let $\cF$ be as above. Then $a_{1,3}=a_{2,4}=a_{2,3}=a_{3,4}=0$.
\end{claim}

\begin{proof}
  Recall that, in view of  Lemma \ref{Omega} we may assume $\HH^1(\cF(-L))=0$.
  Let us summarize the vanishing conditions we have so far by writing
  down the matrix $(a_{i,j})$. Traditionally one rather writes the table $(b_{i,j})=(a_{5-i,5-j})$:
  \[
  \begin{array}{|c|c|c|c|c|c|}
    \hline
    \cF(-2)[-2] & \cF(-F-2L)[-2] & \cF\otimes \Omega_\pi(-F)[-1] & \cF(-1)[-1] & \cF(-L) & \cF \\
    \hline
    \hline
    a_{5,5} & a_{5,4} & 0 & 0 & 0 & 0 \\
    \hline
    0 & a_{4,4} & a_{4,3} & a_{4,2} & 0 & 0 \\
    \hline
    0 & a_{3,4} & a_{3,3} & 0 & a_{3,1} & a_{3,0} \\
    \hline
    0 & a_{2,4} & a_{2,3} & 0 & a_{2,1} & 0 \\
    \hline
    0 & 0 & a_{1,3} & 0 & 0 & 0  \\
    \hline
    0 & 0 & 0 & 0 & 0 & 0\\
    \hline
    \hline
    \cO_X(-L) & \cO_X(F-L) & \cO_X(-F) & \cO_X(L-F) & \Omega_\pi(L) & \cO_X \\
    \hline
  \end{array}
  \]

  This table means that the $(p,q)$-th term of $E_1^{p,q}$ is the
  direct sum of as many copies of $\cE_i$ as the coefficient $(b_{i,j})$ appearing in the
  above table.
  Also, the coefficients above are obtained by computing the dimension of the
  cohomology of the bundle appearing on the $i$-th column of the first
  row, reading cohomological degree from bottom to top, with a shift
  indicated by the brackets.

  Let us focus on the summand $\cO_X(F-L)^{\oplus
    a_{2,4}}=E_1^{1,2}$. By \eqref{nomaps}
  we have $\delta_1^{1,2}=0$. Obviously $\delta_r^{1,2}=0$ for $r \ge
  2$. Also, $E_1^{p,q}=0$ for $p+q \le 2$, so
  $\cO_X(F-L)^{\oplus
    a_{2,4}}$ survives to $E_\infty^{1,2}$, which in turn is zero
  because $E_\infty^{p,q}$ is concentrated at $p+q=5$.
  Therefore $a_{2,4}=0$.
  By the same reason we get $a_{1,3}=0$. Summing up, $E_1^{p,q}=0$ for $p+q \le 3$.

  Let us now look at the summand $\cO_X(-F)^{\oplus
    a_{2,3}}=E_1^{2,2}$. The map $\delta_r^{2,2}$ is clearly zero for
  all $r \ge 1$, and since we proved $E_1^{p,q}=0$ for $p+q \le 3$ we
  get $a_{2,3}=0$ again because  $E_\infty^{p,q}$ is concentrated at $p+q=5$.
  The last vanishing $a_{3,4}=0$ follows a similar pattern.
\end{proof}

In terms of the Beilinson spectral sequence, the previous claim shows
$E_1^{p,q}=0$ for $p+q \le 4$.
Because of \eqref{nomaps}, we have $\delta_r^{2,3}=0$ for all $r \ge
1$, so the vanishing of $E_1^{p,q}$ with $p+q \le 4$ implies that the
term $\cO_X(-F)^{\oplus a_{3,3}}$ survives at $E_\infty^{2,3}$ and is thus
a direct summand of $\gr(\cF)$.
By the same reason, $\cO_X(F-L)^{\oplus a_{4,4}}$ survives at $E_\infty^{1,4}$.
This means that the filtration of $\cF$ induced by the Beilinson-type
spectral sequence takes the form:
\begin{align}
\label{filtra}& 0 = \cF_6 \subset \cF_5 \subset \cdots \subset \cF_0=\cF, &&
                                                              \mbox{with:}
  && \cF_5=\cF_4=\cF_3=0, \\
\nonumber &&&&& \cF_2 \simeq \cO_X(-F)^{\oplus a_{3,3}}, \\
\nonumber &&&&& \cF_1/\cF_2 \simeq \cO_X(F-L)^{\oplus a_{4,4}}, \\
\nonumber &&&&& \cF/\cF_1 \simeq E_\infty^{0,5}.
\end{align}

Our next goal is to compute $E_\infty^{0,5}$.

\begin{lemma} \label{esatta}
  There is an exact sequence:
  \begin{equation}
    \label{esattasuc}
  0 \to E_\infty^{0,5} \to \cG \to \Omega_\pi(L)^{\oplus a_{2,1}} \to 0,
  \end{equation}
  where $\cG$ is a coherent sheaf on $X$ fitting into a long exact sequence:
  \begin{equation}
    \label{lunghissima}
  \begin{split}
  0 \to \cG &\to \Ker(\delta_1^{0,5}) \to \Ker(\delta_1^{2,4}) \to \Ker(\delta_1^{4,3}) \to\\
   &\to \coker(\delta_1^{0,5}) \to \coker(\delta_1^{2,4}) \to \coker(\delta_1^{4,3}) \to 0.
  \end{split}
  \end{equation}
\end{lemma}

Before going into the proof, let us display the maps $\delta_1^{p,q}$
we are interested in:
\begin{align}
  \label{05} & \delta_1^{0,5} : \cO_X(-L)^{\oplus a_{5,5}} \to \cO_X(F-L)^{\oplus  a_{5,4}}, \\
  \label{24} & \delta_1^{2,4} : \cO_X(-F)^{\oplus a_{4,3}} \to \cO_X(L-F)^{\oplus  a_{4,2}}, \\
  \label{43} & \delta_1^{4,3} : \Omega_\pi(L)^{\oplus a_{3,1}} \to \cO_X^{\oplus  a_{3,0}}.
\end{align}

\begin{proof}
We rewrite the cohomology table $(b_{i,j})$ in view of the vanishing proved
in the previous claim and after removing $a_{3,3}$ and $a_{4,4}$ which
do not contribute to $E_\infty^{0,5}$ as we have just seen.
  \[
  \begin{array}{|c|c|c|c|c|c|}
    \hline
    \cF(-2)[-2] & \cF(-F-2L)[-2] & \cF\otimes \Omega_\pi(-F)[-1] & \cF(-1)[-1] & \cF(-L) & \cF \\
    \hline
    \hline
    a_{5,5} & a_{5,4} & 0 & 0 & 0 & 0 \\
    \hline
    0 & 0 & a_{4,3} & a_{4,2} & 0 & 0 \\
    \hline
    0 & 0 & 0 & 0 & a_{3,1} & a_{3,0} \\
    \hline
    0 & 0 & 0 & 0 & a_{2,1} & 0 \\
    \hline
    0 & 0 & 0 & 0 & 0 & 0  \\
    \hline
    0 & 0 & 0 & 0 & 0 & 0\\
    \hline
    \hline
    \cO_X(-L) & \cO_X(F-L) & \cO_X(-F) & \cO_X(L-F) & \Omega_\pi(L) & \cO_X \\
    \hline
  \end{array}
  \]

  In view of this table, we see that the differential $\delta_1$ has
  only three possibly
  non-zero terms, namely $\delta_1^{0,5}$, $\delta_1^{2,4}$ and
  $\delta_1^{4,3}$. So $E_2^{p,q}$ differs from $E_1^{p,q}$ only when
  $(p,q)$ equals $(0,5)$, $(1,5)$, $(2,4)$, $(3,4)$, $(4,3)$ and
  $(5,3)$, and we get:
  \begin{align*}
  & E_2^{0,5}\simeq \ker(\delta_1^{0,5}), && E_2^{1,5}\simeq \coker(\delta_1^{0,5}),\\
  & E_2^{2,4}\simeq \ker(\delta_1^{2,4}), && E_2^{3,4}\simeq \coker(\delta_1^{2,4}),\\
  & E_2^{4,3}\simeq \ker(\delta_1^{4,3}), && E_2^{5,3}\simeq \coker(\delta_1^{4,3}).
  \end{align*}

  Now, since $E_\infty^{p,q}$ is concentrated at $p+q=5$, we realize that actually
  $E_3^{5,3}=0$ so the map $\delta_2^{3,4} : E_2^{3,4} \to E_2^{5,3}$ is surjective and
  actually also $E_3^{3,4}=0$, hence the kernel of $\delta_2^{3,4}$ is the image of $\delta_2^{0,5}$. We have thus proved the
  second line of \eqref{lunghissima}. By the same reason we have the exactness
  of the sequence:
  \begin{equation}
    \label{pezzo}
  \Ker(\delta_1^{0,5}) \to \Ker(\delta_1^{2,4}) \to \Ker(\delta_1^{4,3}),
  \end{equation}
  where the maps are just $\delta_2^{0,5}$ and  $\delta_2^{2,4}$.

  This completes the analysis of the second page. We turn now to
  $E_3$. Note that $E_3^{1,5}\simeq \Ker(\delta_2^{1,5})$ is the kernel of
  the map $\delta_2^{1,5} : \coker(\delta_1^{0,5}) \to \coker(\delta_1^{2,4})$
  appearing in \eqref{lunghissima}. Similarly $E_3^{4,3}\simeq
  \coker(\delta_2^{2,4})$ is the cokernel of the map $\delta_2^{2,4} : \Ker(\delta_1^{2,4}) \to
  \Ker(\delta_1^{4,3})$ showing up in \eqref{pezzo}. Since
  $E_4^{1,5}=E_4^{4,3}=0$, $\delta_3^{1,5}$ gives
  an isomorphism of $E_3^{1,5}$ to $E_3^{4,3}$, hence the exactness of
  \eqref{lunghissima} is proved at the connecting map between the two rows.

  Finally $E_3^{0,5}\simeq \Ker(\delta_2^{0,5})$ is the kernel $\cG$ of the
  first map appearing in \eqref{pezzo} and clearly $E_3^{0,5} \simeq E_4^{0,5}$. The
  map $\delta_4^{0,5}$ thus sends this kernel surjectively onto
  $E_4^{4,2} \simeq \Omega_\pi(L)^{\oplus a_{2,1}}$, with kernel
  $E_5^{0,5} \simeq E_\infty^{0,5}$.
  The lemma is thus proved.
\end{proof}

\begin{lemma}
  In the previous setting, we have:
  \[
  \Ext^1_X(\cG,\cO_X(F-L))=  \Ext^1_X(\cG,\cO_X(-F))=0.
  \]
\end{lemma}

\begin{proof}
  We use the exact sequence \eqref{lunghissima}.
  Indeed, let $\cN$ be one of the two line bundles $\cO_X(F-L)$ or
  $\cO_X(-F)$ and  apply $\Hom_X(-,\cN)$ to \eqref{lunghissima}. Set $\cG_i$
  for the image of the $i$-th map
  $\delta_2^{2i-2,6-i}$ of \eqref{lunghissima}. Then our
  statement is proved if we show that:
  \begin{equation}
    \label{sannullano}
  \Ext^i_X(\Ker(\delta_1^{2i-2,6-i}),\cN)=0, \qquad  \mbox{for $i=1,2,3$}.
  \end{equation}
  Indeed, this would imply $\Ext^{i+1}_X(\cG_{i},\cN)=0$ for $i=1,2$
  which in turn would give $\Ext^1_X(\cG,\cN)=0$, which is our
  statement.

  To check \eqref{sannullano} we look more closely at the defining maps
  \eqref{05},   \eqref{24} and   \eqref{43}.
  For $i=1$, we note that \eqref{05} is constant along the factor
  $\PP^2$ of the product $X \simeq \PP^1 \times \PP^2$ so
  $\ker(\delta_1^{0,5})$ is the pull-back to $X$ of a  torsion-free
  sheaf on $\PP^1$, twisted by $\cO_X(-L)$. Such sheaf is then locally
  free on $\PP^1$ and therefore splits as a direct sum of line bundles. Actually,
  the form of \eqref{05} implies that there are integers $c_j$, one
  for each $j \in \NN$ (with only finitely many values of $j \in \NN$
  satisfying $c_j \ne 0$) such that:
  \[
  \ker(\delta_1^{0,5}) \simeq \bigoplus_{j \in \NN}
  \cO_X(-L-j F)^{\oplus c_j}.
  \]
  It follows plainly that $\Ext^1_X(\ker(\delta_1^{0,5}),\cN)=0$
  for our choices of $\cN$.

  \medskip
  For $i=2$, applying a similar argument to \eqref{24} we get that there exists a
  torsion-free sheaf $\cV$ on $\PP^2$ such that:
  \begin{equation}
    \label{P2}
  \ker(\delta_1^{2,4}) \simeq \sigma^*(\cV) \otimes \cO_X(-F), \qquad \HH^0(\PP^2,\cV(-1))=0.
  \end{equation}
  Therefore, by Künneth's formula we have:
  \[
  \Ext^2_X(\ker(\delta_1^{2,4}),\cO_X(F-L)) \simeq
  \Ext^2_{\PP^2}(\cV,\cO_{\PP^2}(-1)) \otimes \HH^0(\PP^1,\cO_{\PP^1}(1)),
  \]
  which vanishes because Serre duality and \eqref{P2} imply:
  \[
  \Ext^2_{\PP^2}(\cV,\cO_{\PP^2}(-1))\simeq
  \Hom_{\PP^2}(\cO_{\PP^2},\cV(-2))^\vee = 0.
  \]
  The vanishing for $\cN=\cO_X(-F)$ is clear.

  \medskip
  For $i=3$, again looking at \eqref{43} we get a torsion-free sheaf $\cW$
  on $\PP^2$ such that:
  \begin{equation}
    \label{pullback}
  \ker(\delta_1^{4,3}) \simeq \sigma^*(\cW).
  \end{equation}
  This time Künneth's formula provides
  $\Ext^3_X(\ker(\delta_1^{4,3}),\cN)=0$ immediately.
\end{proof}

\begin{lemma}
  For any sheaf $\cU$ which is an extension of copies of $\cO_X(F-L)$
  and $\cO_X(-F)$, we have $\Ext^1_X(E_\infty^{0,5},\cU)=0$.
\end{lemma}

\begin{proof}
  Clearly, it suffices to check that $\Ext^1_X(E_\infty^{0,5},\cN)=0$,
  with $\cN=\cO_X(-F)$ and $\cN=\cO_X(F-L)$. According to Lemma
  \ref{esatta}, we need to check $\Ext^1_X(\cG,\cN)=0$ and
  $\Ext^2_X(\Omega_\pi(L),\cN)=0$. The first vanishing comes from the
  previous lemma and the second one is straightforward.
\end{proof}

Now comes the key point. Indeed, the sheaf $\cF_1$ taken from the
filtration \eqref{filtra} is an Ulrich sheaf of the form
$\cU$ as in the previous lemma.
Therefore, $\cF$ is the direct sum of $E_\infty^{0,5}$ and
$\cF_1$. But $\cF$ is indecomposable, hence either
$\cF_1=0$ and $\cF \simeq E_\infty^{0,5}$, or $\cF \simeq \cU$. In the
latter case Theorem \ref{classifica} is proved, so it remains to
analyze the former one. So we assume from now on $\cF \simeq E_\infty^{0,5}$.

\begin{lemma}
  The sheaf $\cF \simeq E_\infty^{0,5}$ is isomorphic to $\cO_X(-L)$ or $\cO_X(-1)$.
\end{lemma}

\begin{proof}
  Since $\cF \simeq E_\infty^{0,5}$ we have
  $a_{3,3}=a_{4,4}=0$ so the cohomology table $(b_{i,j})$ looks as in
  the proof of Lemma \ref{esatta}.
  We argue now on whether $\HH^0(\cF(L))$ is zero or not.

  If $\HH^0(\cF(L)) \ne 0$, looking at the cohomology table of $\cF$ we
  see that $\HH^1(\cF(-L))=\HH^2(\cF(-F-2L))=0$, and because we are
  assuming $\HH^0(\cF)=0$, we have that item
  \ref{isOL} of Lemma \ref{uovo di colombo} applies to give $\cF
  \simeq \cO_X(-L)$.

  If $\HH^0(\cF(L)) =0$, we use once more the
  vertical Euler sequence, in the form:
  \[
  0 \to \Omega_\pi(-F) \to \cO_X(-1)^{\oplus 3} \to \cO_X(-F) \to 0.
  \]
  We tensor this sequence with $\cF$ and take cohomology.
  From the cohomology table of $\cF$ we
  extract $\HH^2(\cF \otimes \Omega_\pi(-F))=0$,
  which combined with the fact that $\cF$ is ACM
  gives $\HH^1(\cF(-F))=0$. Also, of course $\HH^2(\cF(-1))=0$, while
  $\HH^0(\cF(1)) \ne 0$ by assumption.
  Therefore, item \ref{isO} of Lemma \ref{uovo di colombo} applies and shows $\cF
  \simeq \cO_X(-1)$.

\end{proof}

This completes the proof of Theorem \ref{classifica}.

\subsubsection{Proof of Corollaries \ref{cor1} and \ref{cor2}}
\label{subsection:rigid}

The proof of Corollary \ref{cor1} goes as follows.
Set $\ru(t)=\frac 12(t+2)(t+1)t$ and
note that $\ru$ is the
reduced Hilbert polynomial of an Ulrich sheaf on $X$ initialized by $t=1$.

Let $\cF$ be a
semistable ACM bundle on $X$. 
According to Theorem \ref{B}, for each indecomposable
direct summand $\cG$ of the
graded bundle $\gr(\cF)$ provided by a Jordan-Hölder filtration of $\cF$, there
is some $s \in \ZZ$ such that $\HH^0(\cG(s))=0$ and
$\HH^0(\cG(s+1))\ne 0$, so $\cG(s)$ is one of the sheaves appearing in the next table, where the reduced Hilbert
polynomial is also shown:
\[
  \begin{tabular}{|c|c|c|c|c|c|}
    \hline
    $\cG$ & $\cO_X(-F)$ & $\cO_X(F-L)$ & $\Omega_\pi(L)$ & $\cO_X$ & $\cO_X(-L)$\\
    \hline
    $\rp(\cG(s))$ & $\ru(t)$ &  $\ru(t)$ &  $\ru(t)$
    & $\frac 12(t+2)(t+1)^2$ & $\frac 12(t+1)^2t$ \\
    \hline
  \end{tabular}
\]

Note that these polynomials are pairwise distinct, even upon replacing $t$ by
$t+s$ for any $s \in \ZZ$. Therefore, there is a fixed $s \in \ZZ$
such that either all the summands $\cG(s+1)$ are Ulrich bundles (in
which case the summands $\cG(s)$ are isomorphic to
$\cO_X(-F)$ or $\cO_X(F-L)$ or $\Omega_\pi(L)$), either all summands
$\cG(s)$ are isomorphic to $\cO_X$, or finally they are isomorphic to
$\cO_X(-L)$. In the last two cases $\cG(s)$ is itself a trivial bundle
or a direct sum of copies of $\cO_X(-L)$. In the first case the
graded bundle $\gr(\cF(-s))$ is of
the form:
\[
\cO_X(-F)^{\oplus a} \oplus \cO_X(F-L)^{\oplus b} \oplus
\Omega_\pi(L)^{\oplus c},
\]
for some integers $(a,b,c)$.
Note that there are finitely many
ways to choose $a,b,c$ in the above display while keeping the
Hilbert polynomial unchanged.
This shows that the moduli space of semistable ACM
bundles with fixed Hilbert polynomial is a finite set. Corollary
\ref{cor1} is proved.

\bigskip

For the proof of Corollary \ref{cor2}, we construct the bundles
$\cU_k$ by mutation. Put:
\begin{align*}
  & \cU_{-1} = \cO_X(-F), \\
  & \cU_0 = \cO_X(F-L), \\
  & \cU_1 = \rL_{\cU_0} \cU_{-1} [1], && \cU_{k+1} = \rL_{\cU_k} \cU_{k-1}, && \mbox{for $k \ge 1$}, \\
  & \cU_{-2} = \rR_{\cU_{-1}} \cU_{0}[-1], && \cU_{-k-2} = \rR_{\cU_{-k-1}} \cU_{-k} && \mbox{for $k \ge 1$}.
\end{align*}

The fact that the objects $\cU_k$ are exceptional sheaves having a
resolution of the desired form follows as in \cite[Theorem B]{faenzi-malaspina:minimal}.

By Theorem \ref{B}, any indecomposable rigid ACM bundle on $X$ must
be, up to a twist, isomorphic to $\cO_X(-1)$, $\cO_X(-L)$ or
$\Omega_\pi(L)$ or a rigid Ulrich bundle of the form \eqref{ulrich-extension}. In turn, again as in
\cite[Theorem B]{faenzi-malaspina:minimal} we have that a rigid sheaf
appearing as the middle term of \eqref{ulrich-extension} must be
isomorphic to $\cU_k$ for some $k \in \ZZ$, with
$(a,b)=(c_{k-1},c_k)$. Moreover the equality $(a,b)=(c_{k-1},c_k)$ determines
the rigid bundle $\cU_k$ uniquely.

Finally, given $k \in \ZZ$, since $\cU_k^\vee \otimes \omega_X(2)$ is
a rigid Ulrich bundle which fits as middle term of an extension of the
form \eqref{ulrich-extension} with the same values of $a$ and $b$ as
$\cU_{1-k}$, by the uniqueness argument for the rigid bundles $\cU_k$
(cf. again \cite[\S 2]{faenzi-malaspina:minimal})
we must have $\cU_{1-k} \simeq \cU_k^\vee \otimes \omega_X(2)$.
This concludes the proof of Corollary \ref{cor2}.

\def\cprime{$'$} \def\cprime{$'$} \def\cprime{$'$} \def\cprime{$'$}
  \def\cprime{$'$} \def\cprime{$'$} \def\cprime{$'$} \def\cprime{$'$}
  \def\cprime{$'$}
\providecommand{\bysame}{\leavevmode\hbox to3em{\hrulefill}\thinspace}
\providecommand{\MR}{\relax\ifhmode\unskip\space\fi MR }
\providecommand{\MRhref}[2]{%
  \href{http://www.ams.org/mathscinet-getitem?mr=#1}{#2}
}
\providecommand{\href}[2]{#2}


\bigskip
\noindent
Daniele Faenzi,\\
Institut de Mathématiques de Bourgogne -- UMR CNRS 5584,\\
Université de Bourgogne et Franche-Comté,\\
9 avenue Alain Savary -- BP 47870,\\
21078 Dijon Cedex, France.\\
e-mail: {\tt daniele.faenzi@u-bourgogne.fr}

\bigskip
\noindent
Francesco Malaspina,\\
Dipartimento di Scienze Matematiche, Politecnico di Torino,\\
c.so Duca degli Abruzzi 24,\\
10129 Torino, Italy\\
e-mail: {\tt francesco.malaspina@polito.it}

\bigskip
\noindent
Giangiacomo Sanna,\\
Institut für Mathematik, Freie Universität Berlin\\
Arnimallee 3,\\ 14195 Berlin, Germany\\
e-mail: {\tt gianciacomo.sanna@gmail.com}

\end{document}